\documentclass{amsart}

 \usepackage[foot]{amsaddr}
\let\oldtocsection=\tocsection
 
\let\oldtocsubsection=\tocsubsection 
 
\let\oldtocsubsubsection=\tocsubsubsection

\renewcommand{\tocsection}[2]{\vspace{0.5em}\hspace{0em}\oldtocsection{#1}{#2}}
\renewcommand{\tocsubsection}[2]{\vspace{0.5em}\hspace{1em}\oldtocsubsection{#1}{#2}}
\renewcommand{\tocsubsubsection}[2]{\vspace{0.5em}\hspace{2em}\oldtocsubsubsection{#1}{#2}}
\usepackage{graphicx,cancel,xcolor,hyperref,comment,graphicx,geometry}

\usepackage{tikz}
\usepackage{graphicx,url,etoolbox}
\usepackage{lipsum}
\usepackage{dsfont}

\usepackage{fancyhdr}
\usepackage{lastpage}

\newtheorem{theoreme}{Theorem}[section]
\newtheorem{pro}[theoreme]{Proposition}
\newtheorem{lemma}[theoreme]{Lemma}

\newtheorem{definition}[theoreme]{Definition}
\theoremstyle{definition}

\numberwithin{equation}{section}

 \renewenvironment{proof}{{\bfseries \noindent Proof.}}{\demo}
\newcommand\xqed[1]{%
  \leavevmode\unskip\penalty9999 \hbox{}\nobreak\hfill
  \quad\hbox{#1}}
\newcommand\demo{\xqed{$\square$}}

\hypersetup{bookmarks, colorlinks, urlcolor=blue, citecolor=blue, linkcolor=blue, hyperfigures, pagebackref,
    pdfcreator=LaTeX, breaklinks=true, pdfpagelayout=SinglePage, bookmarksopen=true,bookmarksopenlevel=2}

\usepackage{comment}
\def\u2{\u^2}
\def\u3{\u^3}
\def\u4{\u^4}
\def\u5{\u^5}
\def\y1{\y^1}
\def\y2{\y^2}
\def\y3{\y^3}
\def\y4{\y^4}
\def\y5{\y^5}

\def\R{\mathbb R}

\def\HH{\mathcal H}

\def\AA{\mathcal A}

\def\la {{\lambda}}

\newcommand {\nc}   {\newcommand}
\nc {\be}   {\begin{equation}} \nc {\ee}   {\end{equation}} \nc
{\beq}  {\begin{eqnarray}} \nc {\eeq}  {\end{eqnarray}} \nc {\beqs}
{\begin{eqnarray*}} \nc {\eeqs} {\end{eqnarray*}}
\def\edc{\end{document}}

\providecommand{\abs}[1]{\lvert#1\rvert}
 
\DeclareMathOperator{\divv}{div}  
\DeclareMathOperator{\meas}{meas}
\usepackage{tikz}
\usetikzlibrary{decorations.pathmorphing,patterns,scopes,intersections,calc}
\usepackage{caption}
         
\begin{document}
\title[\fontsize{7}{9}\selectfont  ]{Stability results of locally coupled wave equations with local Kelvin-Voigt damping: Cases when the supports of damping and coupling coefficients are disjoint}
\author{Mohammad Akil$^{1}$, Haidar Badawi$^{1}$, and Serge Nicaise$^1$}

\address{$^1$ Universit\'e Polytechnique Hauts-de-France, CERAMATHS/DEMAV,
	Valenciennes, France}
\email{Mohammad.Akil@uphf.fr, Haidar.Badawi@uphf.fr, Serge.Nicaise@uphf.fr}


\setcounter{equation}{0}
\begin{abstract}
In this paper, we study the direct/indirect stability of locally coupled wave equations with local Kelvin-Voigt dampings/damping and by assuming that the supports of the dampings and the coupling coefficients are disjoint. First, we prove the well-posedness, strong stability, and polynomial stability for some one dimensional coupled systems. Moreover, under some geometric control condition, we prove  the well-posedness and strong stability in the multi-dimensional case.

\end{abstract}
\maketitle
\pagenumbering{roman}
\maketitle
\tableofcontents
\pagenumbering{arabic}
\setcounter{page}{1}
\section{Introduction} 
\noindent  
The direct and indirect stability of locally coupled wave equations with local damping arouses many interests in recent years. The study of coupled systems is also motivated by several physical considerations like Timoshenko and Bresse systems (see for instance \cite{BASSAM20151177,Akil2020,Akil2021,ABBresse,Wehbe08,Fatori01,FATORI2012600}).
 The exponential or polynomial stability of the wave equation with a local Kelvin-Voigt damping is considered in \cite{Liu-Rao:06,Tebou:16,BurqSun:22}, for instance. On the other hand, the direct and indirect stability of locally and   coupled wave equations with local viscous dampings are analyzed in \cite{Alabau-Leautaud:13,Kassemetal:19,Gerbietal:21}. In this paper, we are interested in locally coupled wave equations with  local Kelvin-Voigt dampings. Before stating our main contributions, let us mention similar results for such systems. In 2019, Hayek {\it et al.} in \cite{Hayek}, studied the stabilization of a multi-dimensional system
of weakly coupled wave equations with one or two locally Kelvin-Voigt damping and
non-smooth coefficient at the interface. They established different stability results. In 2021, Akil {\it et al.} in \cite{Wehbe2021}, studied the stability of an elastic/viscoelastic transmission problem of locally coupled waves with non-smooth coefficients, by considering:
\begin{equation*} \left\{	\begin{array}{llll}\vspace{0.15cm}
	\displaystyle  	u_{tt}-\left(au_x +{\color{black}b_0 \chi_{(\alpha_1,\alpha_3)}} {\color{black}u_{tx}}\right)_x +{\color{black}c_0 \chi_{(\alpha_2,\alpha_4)}}y_t =0,& \text{in}\ (0,L)\times (0,\infty) ,&\\ \vspace{0.15cm}
		y_{tt}-y_{xx}-{\color{black}c_0 \chi_{(\alpha_2,\alpha_4)}}u_t =0,  &\text{in} \  (0,L)\times (0,\infty) ,&\\\vspace{0.15cm}
		u(0,t)=u(L,t)=y(0,t)=y(L,t)=0,& \text{in} \ (0,\infty) ,&
	\end{array}\right.
\end{equation*}
where $a, b_0, L >0$, $c_0 \neq 0$, and $0<\alpha_1<\alpha_2<\alpha_3<\alpha_4<L$. They established a polynomial energy decay rate of type $t^{-1}$. In the same year, Akil {\it et al.} in \cite{ABWdelay}, studied the stability of a singular local interaction elastic/viscoelastic coupled wave equations with time delay, by considering:
\begin{equation*} \left\{	\begin{array}{llll}\vspace{0.15cm}
		\displaystyle  	u_{tt}-\left[au_x +{\color{black} \chi_{(0,\beta)}}(\kappa_1 {\color{black}u_{tx}}+\kappa_2 u_{tx}(t-\tau))\right]_x +{\color{black}c_0 \chi_{(\alpha,\gamma)}}y_t =0,& \text{in}\ (0,L)\times (0,\infty) ,&\\ \vspace{0.15cm}
		y_{tt}-y_{xx}-{\color{black}c_0 \chi_{(\alpha,\gamma)}}u_t =0,  &\text{in} \  (0,L)\times (0,\infty) ,&\\\vspace{0.15cm}
		u(0,t)=u(L,t)=y(0,t)=y(L,t)=0,& \text{in} \ (0,\infty) ,&
	\end{array}\right.
\end{equation*}
where $a, \kappa_1, L>0$, $\kappa_2, c_0 \neq 0$, and $0<\alpha <\beta <\gamma <L$. They proved that the energy of their system decays polynomially in $t^{-1}$. In 2021, Akil {\it et al.} in \cite{ABNWmemory}, studied the stability of coupled wave models with locally memory in a past history framework via non-smooth coefficients on the interface, by considering:
\begin{equation*} \left\{	\begin{array}{llll}\vspace{0.15cm}
		\displaystyle  	u_{tt}-\left(au_x +{\color{black} b_0 \chi_{(0,\beta)}} {\color{black}\int_0^{\infty}g(s)u_{x}(t-s)ds}\right)_x +{\color{black}c_0 \chi_{(\alpha,\gamma)}}y_t =0,& \text{in}\ (0,L)\times (0,\infty) ,&\\ \vspace{0.15cm}
		y_{tt}-y_{xx}-{\color{black}c_0 \chi_{(\alpha,\gamma)}}u_t =0,  &\text{in} \  (0,L)\times (0,\infty) ,&\\\vspace{0.15cm}
		u(0,t)=u(L,t)=y(0,t)=y(L,t)=0,& \text{in} \ (0,\infty) ,&
	\end{array}\right.
\end{equation*}
where $a, b_0, L >0$, $c_0 \neq 0$, $0<\alpha<\beta<\gamma<L$, and $g:[0,\infty) \longmapsto (0,\infty)$ is the convolution kernel function. They established an exponential energy decay rate if the two waves have the same speed of propagation. In case of different speed of propagation, they proved that the energy of their system decays polynomially with rate $t^{-1}$. In the same year, Akil {\it et al.} in \cite{akil2021ndimensional}, studied the stability of a multi-dimensional elastic/viscoelastic transmission problem with Kelvin-Voigt damping and non-smooth coefficient at the interface, they established some polynomial stability results under some geometric control condition. In those previous literature, the authors deal with the locally coupled wave equations with local damping and by assuming that there is an intersection between the damping and coupling regions. The aim of this paper is to study the direct/indirect stability of locally coupled wave equations with Kelvin-Voigt dampings/damping localized via non-smooth coefficients/coefficient  and by assuming that the supports of the dampings and coupling coefficients are disjoint. In the first part of this paper, we consider the following one dimensional coupled system:
\begin{eqnarray}
u_{tt}-\left(au_x+bu_{tx}\right)_x+c y_t&=&0,\quad (x,t)\in (0,L)\times (0,\infty),\label{eq1}\\ 
y_{tt}-\left(y_x+dy_{tx}\right)_x-cu_t&=&0,\quad (x,t)\in (0,L)\times (0,\infty),\label{eq2}
\end{eqnarray}
with fully Dirichlet boundary conditions, 
\begin{equation}\label{eq3}
u(0,t)=u(L,t)=y(0,t)=y(L,t)=0,\  t\in (0,\infty),
\end{equation}
and the following initial conditions
\begin{equation}\label{eq4}
	u(\cdot,0)=u_0(\cdot),\ u_t(\cdot,0)=u_1(\cdot),\ y(\cdot,0)=y_0(\cdot)\quad \text{and}\quad y_t(\cdot,0)=y_1(\cdot), \ x \in (0,L).
\end{equation}
In this part, for all $b_0, d_0 >0$ and $c_0 \neq 0$, we treat the following three cases:\\[0.1in]
\textbf{Case 1 (See Figure \ref{p7-Fig1}):}
\begin{equation}\tag{${\rm C1}$}\label{C1}
\left\{\begin{array}{l}
b(x)=b_0 \chi_{(b_1,b_2)}(x)
,\ \quad c(x)=c_0\chi_{(c_1,c_2)}(x),\ \quad 
d(x)=d_0\chi_{(d_1,d_2)}(x),\\[0.1in]
\text{where}\  0<b_1<b_2<c_1<c_2<d_1<d_2<L.	
\end{array}
\right.
\end{equation}
\textbf{Case 2 (See Figure \ref{p7-Fig2}):}
\begin{equation}\tag{${\rm C2}$}\label{C2}
\left\{\begin{array}{l}
b(x)=b_0 \chi_{(b_1,b_2)}(x)
,\ \quad c(x)=c_0\chi_{(c_1,c_2)}(x),\ \quad 
d(x)=d_0\chi_{(d_1,d_2)}(x),\\[0.1in]
\text{where}\ 0<b_1<b_2<d_1<d_2<c_1<c_2<L.
\end{array}
\right.
\end{equation}
\textbf{Case 3 (See Figure \ref{p7-Fig3}):}
\begin{equation}\tag{${\rm C3}$}\label{C3}
\left\{\begin{array}{l}
b(x)=b_0 \chi_{(b_1,b_2)}(x)
,\ \quad c(x)=c_0\chi_{(c_1,c_2)}(x),\ \quad 
d(x)=0,\\[0.1in]
\text{where}\ 0<b_1<b_2<c_1<c_2<L. 
\end{array}
\right.
\end{equation}
\begin{figure}[h!]
\begin{center}
\begin{tikzpicture}[scale=0.8]
\draw[->](0,0)--(8,0);
\draw[->](0,0)--(0,4);
\draw[dashed](1,0)--(1,1);
\draw[dashed](2,0)--(2,1);
\draw[dashed](3,0)--(3,2);
\draw[dashed](4,0)--(4,2);
\draw[dashed](5,0)--(5,3);
\draw[dashed](6,0)--(6,3);
\node[black,below] at (1,0){\scalebox{0.75}{$b_1$}};
\node at (1,0) [circle, scale=0.3, draw=black!80,fill=black!80] {};
\node[black,below] at (2,0){\scalebox{0.75}{$b_2$}};
\node at (2,0) [circle, scale=0.3, draw=black!80,fill=black!80] {};
\node[black,below] at (3,0){\scalebox{0.75}{$c_1$}};
\node at (3,0) [circle, scale=0.3, draw=black!80,fill=black!80] {};
\node[black,below] at (4,0){\scalebox{0.75}{$c_2$}};
\node at (4,0) [circle, scale=0.3, draw=black!80,fill=black!80] {};
\node[black,below] at (5,0){\scalebox{0.75}{$d_1$}};
\node at (5,0) [circle, scale=0.3, draw=black!80,fill=black!80] {};
\node[black,below] at (6,0){\scalebox{0.75}{$d_2$}};
\node at (6,0) [circle, scale=0.3, draw=black!80,fill=black!80] {};
\node[black,below] at (7,0){\scalebox{0.75}{$L$}};
\node at (7,0) [circle, scale=0.3, draw=black!80,fill=black!80] {};
\node[red,left] at (0,1){\scalebox{0.75}{$b_0$}};
\node at (0,1) [circle, scale=0.3, draw=black!80,fill=black!80] {};
\node[blue,left] at (0,2){\scalebox{0.75}{$c_0$}};
\node at (0,2) [circle, scale=0.3, draw=black!80,fill=black!80] {};

\node[black,below] at (0,0){\scalebox{0.75}{$0$}};
\node at (0,0) [circle, scale=0.3, draw=black!80,fill=black!80] {};
\node[green,left] at (0,3){\scalebox{0.75}{$d_0$}};
\node at (0,3) [circle, scale=0.3, draw=black!80,fill=black!80] {};
\draw[-,red](1,1)--(2,1);
\draw[-,blue](3,2)--(4,2);
\draw[-,green](5,3)--(6,3);
\end{tikzpicture}
\end{center}
\caption{Geometric description of the functions $b, c$ and $d$ in Case 1.}\label{p7-Fig1}
\end{figure}
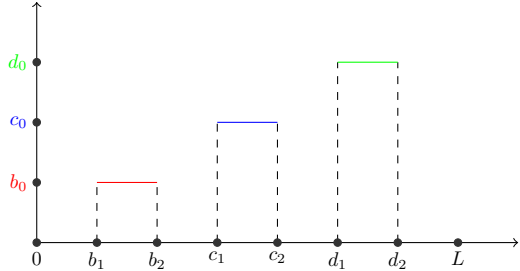
\begin{figure}[h!]
\begin{center}
\begin{tikzpicture}[scale=0.8]
\draw[->](0,0)--(8,0);
\draw[->](0,0)--(0,4);
\draw[dashed](1,0)--(1,1);
\draw[dashed](2,0)--(2,1);
\draw[dashed](3,0)--(3,2);
\draw[dashed](4,0)--(4,2);
\draw[dashed](5,0)--(5,3);
\draw[dashed](6,0)--(6,3);
\node[black,below] at (1,0){\scalebox{0.75}{$b_1$}};
\node at (1,0) [circle, scale=0.3, draw=black!80,fill=black!80] {};
\node[black,below] at (2,0){\scalebox{0.75}{$b_2$}};
\node at (2,0) [circle, scale=0.3, draw=black!80,fill=black!80] {};

\node[black,below] at (0,0){\scalebox{0.75}{$0$}};
\node at (0,0) [circle, scale=0.3, draw=black!80,fill=black!80] {};
\node[black,below] at (3,0){\scalebox{0.75}{$d_1$}};
\node at (3,0) [circle, scale=0.3, draw=black!80,fill=black!80] {};
\node[black,below] at (4,0){\scalebox{0.75}{$d_2$}};
\node at (4,0) [circle, scale=0.3, draw=black!80,fill=black!80] {};
\node[black,below] at (5,0){\scalebox{0.75}{$c_1$}};
\node at (5,0) [circle, scale=0.3, draw=black!80,fill=black!80] {};
\node[black,below] at (6,0){\scalebox{0.75}{$c_2$}};
\node at (6,0) [circle, scale=0.3, draw=black!80,fill=black!80] {};
\node[black,below] at (7,0){\scalebox{0.75}{$L$}};
\node at (7,0) [circle, scale=0.3, draw=black!80,fill=black!80] {};
\node[red,left] at (0,1){\scalebox{0.75}{$b_0$}};
\node at (0,1) [circle, scale=0.3, draw=black!80,fill=black!80] {};
\node[green,left] at (0,2){\scalebox{0.75}{$d_0$}};
\node at (0,2) [circle, scale=0.3, draw=black!80,fill=black!80] {};
\node[blue,left] at (0,3){\scalebox{0.75}{$c_0$}};
\node at (0,3) [circle, scale=0.3, draw=black!80,fill=black!80] {};
\draw[-,red](1,1)--(2,1);
\draw[-,green](3,2)--(4,2);
\draw[-,blue](5,3)--(6,3);
\end{tikzpicture}
\end{center}
\caption{Geometric description of the functions $b,c$ and $d$ in Case 2.}\label{p7-Fig2}
\end{figure}
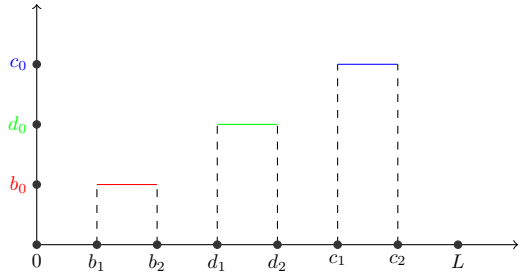
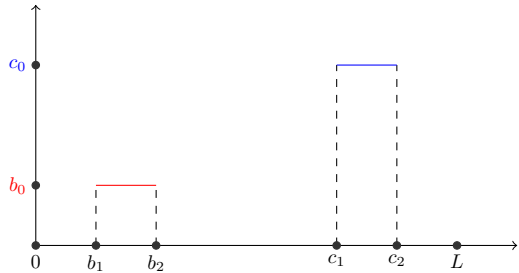
\begin{figure}[h!]
\begin{center}
\begin{tikzpicture}[scale=0.8]
\draw[->](0,0)--(8,0);
\draw[->](0,0)--(0,4);
\draw[dashed](1,0)--(1,1);
\draw[dashed](2,0)--(2,1);
\draw[dashed](5,0)--(5,3);
\draw[dashed](6,0)--(6,3);

\node[black,below] at (0,0){\scalebox{0.75}{$0$}};
\node at (0,0) [circle, scale=0.3, draw=black!80,fill=black!80] {};

\node[black,below] at (1,0){\scalebox{0.75}{$b_1$}};
\node at (1,0) [circle, scale=0.3, draw=black!80,fill=black!80] {};
\node[black,below] at (2,0){\scalebox{0.75}{$b_2$}};
\node at (2,0) [circle, scale=0.3, draw=black!80,fill=black!80] {};
\node[black,below] at (5,0){\scalebox{0.75}{$c_1$}};
\node at (5,0) [circle, scale=0.3, draw=black!80,fill=black!80] {};
\node[black,below] at (6,0){\scalebox{0.75}{$c_2$}};
\node at (6,0) [circle, scale=0.3, draw=black!80,fill=black!80] {};
\node[black,below] at (7,0){\scalebox{0.75}{$L$}};
\node at (7,0) [circle, scale=0.3, draw=black!80,fill=black!80] {};
\node[red,left] at (0,1){\scalebox{0.75}{$b_0$}};
\node at (0,1) [circle, scale=0.3, draw=black!80,fill=black!80] {};
\node[blue,left] at (0,3){\scalebox{0.75}{$c_0$}};
\node at (0,3) [circle, scale=0.3, draw=black!80,fill=black!80] {};
\draw[-,red](1,1)--(2,1);
\draw[-,blue](5,3)--(6,3);
\end{tikzpicture}
\end{center}
\caption{Geometric description of the functions $b$ and $c$ in Case 3.}\label{p7-Fig3}
\end{figure}
\noindent While in the second part, we consider the following multi-dimensional coupled system:
\begin{eqnarray}\label{ND-1}
u_{tt}-\divv (\nabla u+bu_t)+cy_t&=&0\quad \text{in}\ \Omega\times (0,\infty),\\
y_{tt}-\Delta y-cy_t&=&0\quad \text{in}\ \Omega\times (0,\infty),
\end{eqnarray}
with full Dirichlet boundary condition 
\begin{equation}\label{ND-2}
u=y=0\quad \text{on}\quad \Gamma\times (0,\infty),
\end{equation}
and the following initial condition 
\begin{equation}\label{ND-5}
	u(\cdot,0)=u_0(\cdot),\ u_t(\cdot,0)=u_1(\cdot),\ y(\cdot,0)=y_0(\cdot)\ \text{and}\ y_t(\cdot,0)=y_1(\cdot) \ \text{in} \ \Omega,
\end{equation}
where $\Omega \subset  \R^d$, $d\geq 2$ is an open and bounded set with   boundary $\Gamma$ of class $C^2$. Here, $b,c\in L^{\infty}(\Omega)$ are such that $b:\Omega\to \R_+$ is the viscoelastic damping coefficient, $c:\Omega\to \R$ is the coupling function and 
\begin{equation}\label{ND-3}
b(x)\geq b_0>0\ \ \text{in}\ \ \omega_b\subset \Omega, \quad c(x)\geq c_0\neq 0\ \ \text{in}\ \ \omega_c\subset \Omega\quad \text{and}\quad c(x)=0\ \ \text{on}\ \ \Omega\backslash \omega_c
\end{equation}
and 
\begin{equation}\label{ND-4}
\meas\left(\overline{\omega_c}\cap \Gamma\right)>0\quad \text{and}\quad \overline{\omega_b}\cap \overline{\omega_c}=\emptyset.
\end{equation}\\\linebreak
 In the first part of this paper, we study the direct and indirect stability of system \eqref{eq1}-\eqref{eq4} by considering   the three cases \eqref{C1}, \eqref{C2}, and \eqref{C3}. In Subsection \ref{WP}, we prove the well-posedness of our system by using a semigroup approach. In Subsection \ref{subss}, by using a general criteria of Arendt-Batty, we prove the strong stability of our system in the absence of the compactness of the resolvent. Finally, in Subsection \ref{secps}, by using a frequency domain approach combined with a specific multiplier method, we prove that our system decay polynomially in $t^{-4}$ or in $t^{-1}$.\\\linebreak 
In the second part of this paper, we study the indirect stability of system \eqref{ND-1}-\eqref{ND-5}. In Subsection \ref{wpnd}, we prove the well-posedness of our system by using a semigroup approach. Finally, in Subsection \ref{Strong Stability-ND}, under some geometric control condition, we prove the strong stability of this system.

\section{Direct and Indirect Stability in the one dimensional case}\label{sec1}
In this section, we study the well-posedness, strong stability, and polynomial stability of system \eqref{eq1}-\eqref{eq4}. The main result of this section are the following three subsections.
\subsection{Well-Posedness}\label{WP}
\noindent In this subsection, we will establish the well-posedness of system \eqref{eq1}-\eqref{eq4} by using  semigroup approach. The energy of system \eqref{eq1}-\eqref{eq4} is given by
	\begin{equation*}
	E(t)=\frac{1}{2}\int_0^L \left(|u_t|^2+a|u_x|^2+|y_t|^2+|y_x|^2\right)dx.
	\end{equation*}
Let $\left(u,u_{t},y,y_{t}\right)$ be a regular solution of \eqref{eq1}-\eqref{eq4}. Multiplying  \eqref{eq1} and \eqref{eq2}    by $\overline{u_t}$ and $\overline{y_t}$  respectively, then  using  the boundary conditions \eqref{eq3}, we get
	\begin{equation*}
	E^\prime(t)=- \int_0^L \left(b|u_{tx}|^2+d|y_{tx}|^2\right)dx.
	\end{equation*}
Thus, if \eqref{C1} or \eqref{C2} or \eqref{C3} holds, we get 	$E^\prime(t)\leq0$. Therefore, system \eqref{eq1}-\eqref{eq4} is dissipative in the sense that its energy is non-increasing with respect to time $t$. Let us define the energy space $\mathcal{H}$ by
	\begin{equation*}
	\mathcal{H}=(H_0^1(0,L)\times L^2(0,L))^2.
	\end{equation*}

\noindent The energy space $\mathcal{H}$ is equipped with the following inner product
	$$
	\left(U,U_1\right)_\mathcal{H}=\int_{0}^Lv\overline{{v}}_1dx+a\int_{0}^Lu_x(\overline{{u}}_1)_xdx+\int_{0}^Lz\overline{{z}}_1dx+\int_{0}^Ly_x(\overline{{y}}_1)_xdx,
	$$
for  all  $U=\left(u,v,y,z\right)^\top$ and $U_1=\left(u_1,v_1,y_1,z_1\right)^\top$ in $\mathcal{H}$.
We define the unbounded linear operator $\mathcal{A}: D\left(\mathcal{A}\right)\subset \mathcal{H}\longrightarrow \mathcal{H}$  by 
\begin{equation*}
D(\mathcal{A})=\left\{\begin{array}{l}
\displaystyle
U=(u,v,y,z)^\top \in\mathcal{H};\ v,z\in H_0^1(0,L), \ (au_{x}+bv_{x})_{x}\in L^2(0,L), \ (y_{x}+dz_x)_x\in L^2(0,L)
\end{array}\right\}
\end{equation*}
and 
$$
\mathcal{A}\left(u, v,y, z\right)^\top=\left(v,(au_{x}+bv_{x})_{x}-cz, z, (y_x+dz_x)_x+cv \right)^{\top}, \ \forall U=\left(u, v,y, z\right)^\top \in D\left(\mathcal{A}\right).
$$
\noindent Now, if $U=(u,u_t,y,y_t)^\top$ is the state  of system \eqref{eq1}-\eqref{eq4}, then it is transformed into the following first order evolution equation 
\begin{equation}\label{eq-2.9}
U_t=\mathcal{A}U,\quad
U(0)=U_0,
\end{equation}
where $U_0=(u_0,u_1,y_0,y_1)^\top \in \HH$.\\\linebreak
\begin{pro}\label{mdissipative}
	{\rm
If \eqref{C1} or \eqref{C2} or \eqref{C3} holds. Then, the unbounded linear operator $\AA$ is m-dissipative in the Hilbert space $\HH$.}
\end{pro}
\begin{proof}
 For all $U=(u,v,y,z)^{\top}\in D(\mathcal{A})$, we have 
\begin{equation*}
\Re\left<\mathcal{A}U,U\right>_{\mathcal{H}}=-\int_0^Lb\abs{v_x}^2dx-\int_0^Ld\abs{z_x}^2dx\leq 0,
\end{equation*}
which implies that $\mathcal{A}$ is dissipative. Now, similiar to Proposition 2.1 in \cite{Wehbe2021} (see also \cite{ABWdelay} and \cite{ABNWmemory}), we can prove that there exists a unique solution  $U=(u,v,y,z)^{\top}\in D(\mathcal{A})$ of 
\begin{equation*}
-\mathcal{A}U=F,\quad \forall F=(f^1,f^2,f^3,f^4)^\top\in \mathcal{H}.
\end{equation*}
Then $0\in \rho(\mathcal{A})$ and $\mathcal{A}$ is an isomorphism and since $\rho(\mathcal{A})$ is open in $\mathbb{C}$ (see Theorem 6.7 (Chapter III) in \cite{Kato01}),  we easily get $R(\lambda I -\mathcal{A}) = {\mathcal{H}}$ for a sufficiently small $\lambda>0 $. This, together with the dissipativeness of $\mathcal{A}$, imply that   $D\left(\mathcal{A}\right)$ is dense in ${\mathcal{H}}$   and that $\mathcal{A}$ is m-dissipative in ${\mathcal{H}}$ (see Theorems 4.5, 4.6 in  \cite{Pazy01}).
\end{proof}\\\linebreak
According to Lumer-Phillips theorem (see \cite{Pazy01}), then the operator $\AA$ generates a $C_{0}$-semigroup of contractions $e^{t\AA}$ in $\HH$ which gives the well-posedness of \eqref{eq-2.9}.
 Then, we have the following result:
\begin{theoreme}{\rm
For all $U_0 \in \HH$,  system \eqref{eq-2.9} admits a unique weak solution $$U(t)=e^{t\AA}U_0\in C^0 (\R_+ ,\HH).
	$$ Moreover, if $U_0 \in D(\AA)$, then the system \eqref{eq-2.9} admits a unique strong solution $$U(t)=e^{t\AA}U_0\in C^0 (\R_+ ,D(\AA))\cap C^1 (\R_+ ,\HH).$$}
\end{theoreme}

\subsection{Strong Stability}\label{subss}
In this subsection, we will prove the strong stability of system \eqref{eq1}-\eqref{eq4}. We define the following conditions: 
\begin{equation}\label{SSC1}\tag{${\rm SSC1}$}
	\eqref{C1} \ \text{holds}\quad \text{and} \quad \abs{c_0}<\min\left(\frac{\sqrt{a}}{c_2-c_1},\frac{1}{c_2-c_1}\right),
\end{equation}
\begin{equation}\label{SSC2}\tag{${\rm SSC3}$}
	\eqref{C3} \ \text{holds},\quad 	a=1\quad \text{and}\quad \abs{c_0}<\frac{1}{c_2-c_1}.
\end{equation}
The main result of this section is the following theorem.
\begin{theoreme}\label{Th-SS1}
{\rm Assume that \eqref{SSC1} or \eqref{C2} or \eqref{SSC2} holds.  Then, the $C_0$-semigroup of contractions $\left(e^{t\mathcal{A}}\right)_{t\geq 0}$ is strongly stable in $\mathcal{H}$; i.e. for all $U_0\in \mathcal{H}$, the solution of \eqref{eq-2.9} satisfies 
$$
\lim_{t\to +\infty}\|e^{t\mathcal{A}}U_0\|_{\mathcal{H}}=0.
$$}
\end{theoreme}
\noindent According to Theorem \ref{App-Theorem-A.2}, to prove Theorem \ref{Th-SS1}, we need to prove that the operator $\AA$ has no pure imaginary eigenvalues and $\sigma(\AA)\cap i\R $ is countable. Its proof has been divided into the following  Lemmas. 
\begin{lemma}\label{ker-SS123}
	{\rm
Assume that  \eqref{SSC1} or \eqref{C2} or  \eqref{SSC2} holds. Then, for all $\la \in \mathbb{R}$, $i\la I-\mathcal{A}$ is injective, i.e. 
$$
\ker\left(i\la I-\mathcal{A}\right)=\left\{0\right\}.
$$}
\end{lemma}
\begin{proof}
From Proposition \ref{mdissipative}, we have $0\in \rho(\mathcal{A})$. We still need to show the result for $\la\in \R^{\ast}$. For this aim, suppose that there exists a real number $\la\neq 0$ and  $U=\left(u,v,y,z\right)^\top\in D(\AA)$ such that
\begin{equation*}
\AA U=i\la U.
\end{equation*}
Equivalently, we have 
\begin{eqnarray}
v&=&i\la u,\label{eq-2.20}\\
(au_{x}+bv_{x})_{x}-cz&=&i\la v,\label{eq-2.21}\\
z&=&i\la y,\label{eq-2.22}\\
(y_{x}+dz_x)+cv&=&i\la z.\label{eq-2.23}
\end{eqnarray}
Next, a straightforward computation gives 
\begin{equation}\label{Re}
0=\Re\left<i\la U,U\right>_{\HH}=\Re\left<\AA U,U\right>_{\HH}=-\int_0^L b|v_x|^2dx-\int_0^L d|z_x|^2dx.
\end{equation}
Inserting \eqref{eq-2.20} and \eqref{eq-2.22} in \eqref{eq-2.21} and \eqref{eq-2.23}, we get
\begin{eqnarray}
\la^2u+(au_{x}+i\la bu_x)_x-i\la cy&=&0\quad \text{in}\quad (0,L),\label{eq-2.27}\\
\la^2y+(y_{x}+i\la dy_x)_x+i\la cu&=&0\quad \text{in}\quad (0,L),\label{eq-2.2.8}
\end{eqnarray}
with the boundary conditions 
\begin{equation}\label{boundaryconditionker}
u(0)=u(L)=y(0)=y(L)=0.
\end{equation}
$\bullet$ \textbf{Case 1:} Assume that \eqref{SSC1} holds.  
 From \eqref{eq-2.20}, \eqref{eq-2.22} and \eqref{Re}, we deduce that 
\begin{equation}\label{2.10}
u_x=	v_x=0 \ \ \text{in} \ \ (b_1,b_2) \ \ \text{and} \ \ y_x=z_x =0 \ \ \text{in} \ \ (d_1,d_2).
\end{equation}
Using \eqref{eq-2.27}, \eqref{eq-2.2.8} and \eqref{2.10}, we obtain
\begin{equation}\label{2interval}
\la^2u+au_{xx}=0\ \ \text{in}\ \ (0,c_1)\quad \text{and}\quad \la^2y+y_{xx}=0\ \ \text{in}\ \ (c_2,L).
\end{equation}
Deriving the above equations with respect to $x$ and using \eqref{2.10}, we get 
\begin{equation}\label{2interval1}
\left\{\begin{array}{lll}
\la^2u_x+au_{xxx}=0&\text{in}&(0,c_1),\\[0.1in]
u_x=0&\text{in}&(b_1,b_2)\subset (0,c_1),
\end{array}
\right.\quad \text{and}\quad 
\left\{\begin{array}{lll}
\la^2y_x+y_{xxx}=0&\text{in}&(c_2,L),\\[0.1in]
y_x=0&\text{in}&(d_1,d_2)\subset (c_2,L).
\end{array}
\right.
\end{equation}
Using the unique continuation theorem, we get 
\begin{equation}\label{2interval2}
u_x=0\ \ \text{in}\ \  (0,c_1)\quad \text{and}\quad y_x=0\ \ \text{in}\ \  (c_2,L)
\end{equation}
Using \eqref{2interval2} and the fact that $u(0)=y(L)=0$, we get 
\begin{equation}\label{2interval3}
u=0\ \ \text{in}\ \  (0,c_1)\quad \text{and}\quad y=0\ \ \text{in}\ \  (c_2,L).
\end{equation}
Now, our aim is to prove that $u=y=0 \ \text{in} \ (c_1,c_2)$. For this aim, using \eqref{2interval3} and the fact that $u, y\in C^1([0,L])$, we obtain the following boundary conditions
\begin{equation}\label{1c1c2}
u(c_1)=u_x(c_1)=y(c_2)=y_x(c_2)=0.
\end{equation}
 Multiplying \eqref{eq-2.27} by $-2(x-c_2)\overline{u}_x$, integrating over $(c_1,c_2)$ and taking the real part, we get 
\begin{equation}\label{ST1step2}
-\int_{c_1}^{c_2}\la^2(x-c_2)(\abs{u}^2)_xdx-a\int_{c_1}^{c_2}(x-c_2)\left(\abs{u_x}^2\right)_xdx+2\Re\left(i\la c_0\int_{c_1}^{c_2}(x-c_2)y\overline{u}_xdx\right)=0,
\end{equation}
using integration by parts and \eqref{1c1c2}, we get 
\begin{equation}\label{ST2step2}
\int_{c_1}^{c_2}\abs{\la u}^2dx+a\int_{c_1}^{c_2}\abs{u_x}^2dx+2\Re\left(i\la c_0\int_{c_1}^{c_2}(x-c_2)y\overline{u}_xdx\right)=0.
\end{equation}
Multiplying \eqref{eq-2.2.8} by $-2(x-c_1)\overline{y}_x$, integrating over $(c_1,c_2)$, taking the real part, and using the same argument as above, we get 
\begin{equation}\label{ST3step2}
\int_{c_1}^{c_2}\abs{\la y}^2dx+\int_{c_1}^{c_2}\abs{y_x}^2dx+2\Re\left(i\la c_0\int_{c_1}^{c_2}(x-c_1)u\overline{y}_x dx\right)=0.
\end{equation}
Adding \eqref{ST2step2} and \eqref{ST3step2}, we get 
\begin{equation}\label{ST4step2}
\int_{c_1}^{c_2}\abs{\la u}^2dx+a\int_{c_1}^{c_2}\abs{u_x}^2dx+\int_{c_1}^{c_2}\abs{\la y}^2dx+\int_{c_1}^{c_2}\abs{y_x}^2dx\leq 2\abs{\la}\abs{c_0}(c_2-c_1)\int_{c_1}^{c_2}\left(\abs{y}\abs{u_x}+\abs{u}\abs{y_x}\right)dx.
\end{equation}
Using Young's inequality in \eqref{ST4step2}, we get 
\begin{equation}\label{ST5step2}
\begin{array}{c}
\displaystyle 
\int_{c_1}^{c_2}\abs{\la u}^2dx+a\int_{c_1}^{c_2}\abs{u_x}^2dx+\int_{c_1}^{c_2}\abs{\la y}^2dx+\int_{c_1}^{c_2}\abs{y_x}^2dx\leq \frac{c_0^2(c_2-c_1)^2}{a}\int_{c_1}^{c_2}\abs{\la y}^2dx\vspace{0.25cm}\\
\displaystyle 
+\, a\int_{c_1}^{c_2}\abs{u_x}^2dx+c_0^2(c_2-c_1)^2\int_{c_1}^{c_2}\abs{\la u}^2dx+\int_{c_1}^{c_2}\abs{y_x}^2dx,
\end{array}
\end{equation}
consequently, we get 
\begin{equation}\label{ST6step2}
\left(1-\frac{c_0^2(c_2-c_1)^2}{a}\right)\int_{c_1}^{c_2}\abs{\la y}^2dx+\left(1-c_0^2(c_2-c_1)^2\right)\int_{c_1}^{c_2}\abs{\la u}^2dx\leq 0.
\end{equation}
Thus, from the above inequality and \eqref{SSC1}, we get 
\begin{equation}\label{0c1c2}
u=y=0 \ \ \text{in} \ \ (c_1,c_2).
\end{equation}
Next, we need to prove that $u=0$ in $(c_2,L)$ and $y=0$ in $(0,c_1)$. For this aim, from \eqref{0c1c2} and the fact that $u,y \in C^1([0,L])$, we obtain
\begin{equation}\label{ST1step3}
u(c_2)=u_x(c_2)=0\quad \text{and}\quad y(c_1)=y_x(c_1)=0. 
\end{equation}
It follows from \eqref{eq-2.27}, \eqref{eq-2.2.8} and \eqref{ST1step3} that 
\begin{equation}\label{ST2step3}
\left\{\begin{array}{lll}
\la^2u+au_{xx}=0\ \ \text{in}\ \ (c_2,L),\\[0.1in]
u(c_2)=u_x(c_2)=u(L)=0,
\end{array}
\right.\quad \text{and}\quad 
\left\{\begin{array}{rcc}
\la^2y+y_{xx}=0\ \ \text{in}\ \ (0,c_1),\\[0.1in]
y(0)=y(c_1)=y_x(c_1)=0.
\end{array}
\right.
\end{equation}
Holmgren uniqueness theorem yields
\begin{equation}\label{2.25}
 u=0 \ \ \text{in} \ \ (c_2,L) \ \ \text{and}  \ \ y=0 \ \ \text{in} \ \  (0,c_1).
\end{equation}
Therefore, from \eqref{eq-2.20}, \eqref{eq-2.22}, \eqref{2interval3}, \eqref{0c1c2} and \eqref{2.25}, we deduce that
$$
U=0.
$$
$\bullet$ \textbf{Case 2:} Assume that  \eqref{C2} holds. From \eqref{eq-2.20}, \eqref{eq-2.22} and \eqref{Re}, we deduce that 
\begin{equation}\label{2.10*}
	u_x=	v_x=0 \ \ \text{in} \ \ (b_1,b_2) \ \ \text{and} \ \ y_x=z_x =0 \ \ \text{in} \ \ (d_1,d_2).
\end{equation}
Using \eqref{eq-2.27}, \eqref{eq-2.2.8} and \eqref{2.10*}, we obtain
\begin{equation}\label{2interval}
\la^2u+au_{xx}=0\ \ \text{in}\ \ (0,c_1)\quad \text{and}\quad \la^2y+y_{xx}=0\ \ \text{in}\ \ (0,c_1).
\end{equation}
Deriving the above equations with respect to $x$ and using \eqref{2.10*}, we get 
\begin{equation}\label{C22interval1}
\left\{\begin{array}{lll}
\la^2u_x+au_{xxx}=0&\text{in}&(0,c_1),\\[0.1in]
u_x=0&\text{in}&(b_1,b_2)\subset (0,c_1),
\end{array}
\right.\quad \text{and}\quad 
\left\{\begin{array}{lll}
\la^2y_x+y_{xxx}=0&\text{in}&(0,c_1),\\[0.1in]
y_x=0&\text{in}&(d_1,d_2)\subset (0,c_1).
\end{array}
\right.
\end{equation}
Using the unique continuation theorem, we get 
\begin{equation}\label{C22interval2}
u_x=0\ \ \text{in}\ \  (0,c_1)\quad \text{and}\quad y_x=0\ \ \text{in}\ \  (0,c_1).
\end{equation}
From \eqref{C22interval2} and the fact that $u(0)=y(0)=0$, we get 
\begin{equation}\label{C22interval3}
u=0\ \ \text{in}\ \  (0,c_1)\quad \text{and}\quad y=0\ \ \text{in}\ \  (0,c_1).
\end{equation}
Using the fact that $u,y\in C^1([0,L])$ and \eqref{C22interval3}, we get 
\begin{equation}\label{C21}
u(c_1)=u_x(c_1)=y(c_1)=y_x(c_1)=0.
\end{equation}
Now, using the definition of $c(x)$ in \eqref{eq-2.27}-\eqref{eq-2.2.8}, \eqref{2.10*} and \eqref{C21} and Holmgren theorem, we get  
 $$u=y=0\ \text{ in} \ (c_1,c_2).$$
  Again, using the fact that $u,y\in C^1([0,L])$, we get 
\begin{equation}\label{C22}
u(c_2)=u_x(c_2)=y(c_2)=y_x(c_2)=0. 
\end{equation}
Now, using the same argument as in Case 1, we obtain 
$$u=y=0 \ \text{in} \  (c_2,L),$$ 
consequently, we deduce that
$$
U=0.
$$
$\bullet$ \textbf{Case 3:} Assume that  \eqref{SSC2} holds.
\noindent Using the same argument as in Cases 1 and 2, we obtain 
\begin{equation}\label{C3-SST1}
u=0\ \ \text{in}\ \ (0,c_1)\quad \text{and}\quad u(c_1)=u_x(c_1)=0.
\end{equation}
\textbf{Step 1.} The aim of this step is to prove that 
\begin{equation}\label{2c1c2}
\int_{c_1}^{c_2}\abs{u}^2dx=\int_{c_1}^{c_2}\abs{y}^2dx. 
\end{equation}
For this aim, multiplying \eqref{eq-2.27} by $\overline{y}$ and \eqref{eq-2.2.8} by $\overline{u}$ and using integration by parts, we get 
\begin{eqnarray}
\int_{0}^{L}\la^2u\overline{y}dx-\int_{0}^{L}u_x\overline{y_x}dx-i\la c_0\int_{c_1}^{c_2}\abs{y}^2dx&=&0,\label{3c1c2}\\[0.1in]
\int_{0}^{L}\la^2y\overline{u}dx-\int_{0}^{L}y_x\overline{u_x}dx+i\la c_0\int_{c_1}^{c_2}\abs{u}^2dx&=&0.\label{4c1c2} 
\end{eqnarray}
Adding \eqref{3c1c2} and \eqref{4c1c2}, taking the imaginary part, we get \eqref{2c1c2}.\\[0.1in] 
\textbf{Step 2.}  
 Multiplying \eqref{eq-2.27} by $-2(x-c_2)\overline{u}_x$, integrating over $(c_1,c_2)$ and taking the real part, we get 
\begin{equation}\label{C3ST3step2}
-\Re\left(\int_{c_1}^{c_2}\la^2(x-c_2)(\abs{u}^2)_xdx\right)-\Re\left(\int_{c_1}^{c_2}(x-c_2)\left(\abs{u_x}^2\right)_xdx\right)+2\Re\left(i\la c_0\int_{c_1}^{c_2}(x-c_2)y\overline{u}_xdx\right)=0,
\end{equation}
using integration by parts in \eqref{C3ST3step2} and \eqref{C3-SST1}, we get 
\begin{equation}\label{C3ST4step2}
\int_{c_1}^{c_2}\abs{\la u}^2dx+a\int_{c_1}^{c_2}\abs{u_x}^2dx+2\Re\left(i\la c_0\int_{c_1}^{c_2}(x-c_2)y\overline{u}_xdx\right)=0.
\end{equation}
Using Young's inequality in \eqref{C3ST4step2}, we obtain 
\begin{equation}\label{C3ST5step2}
\int_{c_1}^{c_2}\abs{\la u}^2dx+\int_{c_1}^{c_2}\abs{u_x}^2dx\leq \abs{c_0}(c_2-c_1)\int_{c_1}^{c_2}\abs{\la y}^2dx+\abs{c_0}(c_2-c_1)\int_{c_1}^{c_2}\abs{u_x}^2dx. 
\end{equation}
Inserting \eqref{2c1c2} in \eqref{C3ST5step2}, we get 
\begin{equation}\label{C3ST6step2}
\left(1-\abs{c_0}(c_2-c_1)\right)\int_{c_1}^{c_2}\left(\abs{\la u}^2+\abs{u_x}^2\right)dx\leq 0.
\end{equation}
According to \eqref{SSC2} and \eqref{2c1c2}, we get 
\begin{equation}\label{C3ST7step2}
u=y=0\quad \text{in}\quad (c_1,c_2).
\end{equation}
\textbf{Step 3.} Using the fact that $u\in H^2(c_1,c_2)\subset C^1([c_1,c_2])$, we get 
\begin{equation}\label{C3ST1step3}
u(c_1)=u_x(c_1)=y(c_1)=y_x(c_1)=y(c_2)=y_x(c_2)=0.
\end{equation}
Now, from \eqref{eq-2.27}, \eqref{eq-2.2.8} and the definition of $c$, we get 
\begin{equation*}
\left\{\begin{array}{lll}
\la^2u+u_{xx}=0\ \ \text{in} \ \ (c_2,L),\\
u(c_2)=u_x(c_2)=0,
\end{array}
\right.\quad \text{and}\quad 
\left\{\begin{array}{lll}
\la^2y+y_{xx}=0\ \ \text{in}\ \ (0,c_1)\cup (c_2,L),\\
y(c_1)=y_x(c_1)=y(c_2)=y_x(c_2)=0.
\end{array}
\right.
\end{equation*}
From the above systems and Holmgren uniqueness Theorem, we get 
\begin{equation}\label{C3ST2step3}
u=0\ \ \text{in}\ \ (c_2,L)\quad \text{and}\quad y=0\ \ \text{in}\ \ (0,c_1)\cup (c_2,L). 
\end{equation}
\\ \noindent Consequently, using \eqref{C3-SST1}, \eqref{C3ST7step2} and \eqref{C3ST2step3}, we get $U=0$. The proof is thus completed.
\end{proof}
\begin{lemma}\label{surjectivity}
	{\rm
Assume that \eqref{SSC1} or \eqref{C2} or  \eqref{SSC2} holds. Then, for all $\lambda\in \mathbb{R}$, we have 
$$
R\left(i\la I-\mathcal{A}\right)=\mathcal{H}.
$$}
\end{lemma}
\begin{proof}
See Lemma 2.5 in \cite{Wehbe2021} (see also \cite{ABNWmemory}).
\end{proof}\\\linebreak 

\noindent \textbf{Proof of Theorems \ref{Th-SS1}}. From Lemma \ref{ker-SS123}, we obtain that the operator $\mathcal{A}$ has no pure imaginary eigenvalues (i.e. $\sigma_p(\AA)\cap i\R=\emptyset$). Moreover, from Lemma \ref{surjectivity} and with the help of the closed graph theorem of Banach, we deduce that $\sigma(\AA)\cap i\R=\emptyset$. Therefore, according to Theorem \ref{App-Theorem-A.2}, we get that the C$_0 $-semigroup $(e^{t\AA})_{t\geq0}$ is strongly stable. The proof is thus complete. \xqed{$\square$}

\subsection{Polynomial Stability}\label{secps}
\noindent In this subsection, we study the polynomial stability of  system \eqref{eq1}-\eqref{eq4}. Our main result in this section are the following theorems. 
\begin{theoreme}\label{1pol}
	{\rm
Assume that \eqref{SSC1} holds. Then, for all $U_0 \in D(\AA)$,  there exists a constant $C>0$ independent of $U_0$ such that
\begin{equation}\label{Energypol1}
E(t)\leq \frac{C}{t^4}\|U_0\|^2_{D(\AA)},\quad t>0.
\end{equation}}
\end{theoreme}
\begin{theoreme}\label{2pol}
	{\rm
Assume that \eqref{SSC2} holds . Then, for all $U_0 \in D(\AA)$ there exists a constant $C>0$ independent of $U_0$ such that 
\begin{equation}\label{Energypol2}
E(t)\leq \frac{C}{t}\|U_0\|^2_{D(\AA)},\quad t>0.
\end{equation}}
\end{theoreme}
\noindent According to Theorem \ref{bt}, the polynomial energy decays \eqref{Energypol1} and \eqref{Energypol2} hold if the following conditions 
\begin{equation}\label{H1}\tag{${\rm{H_1}}$}
i\R\subset \rho(\mathcal{A})
\end{equation}
and
\begin{equation}\label{H2}\tag{${\rm{H_2}}$}
\limsup_{\la \in \R, \ |\la| \to \infty}\frac{1}{|\la|^\ell}\left\|(i\la I-\AA)^{-1}\right\|_{\mathcal{L}(\mathcal{H})}<\infty \ \ \text{with} \ \ \ell=\left\{\begin{array}{lll}
\frac{1}{2} \ \ \text{for Theorem \ref{1pol}},\vspace{0.15cm}\\
2 \ \ \text{for Theorem \ref{2pol}},
\end{array}\right.
\end{equation}
are satisfied. Since condition \eqref{H1} is already proved in Subsection \ref{subss}. We still need to prove \eqref{H2}, let us prove it by a contradiction argument. To this aim, suppose that \eqref{H2} is false,  then there exists $\left\{\left(\la_n,U_n:=(u_n,v_n,y_n,z_n)^\top\right)\right\}_{n\geq 1}\subset \R^{\ast}_+\times D(\AA)$ with 
\begin{equation}\label{pol1}
\la_n\to \infty \ \text{as} \ n\to \infty \quad \text{and}\quad \|U_n\|_{\mathcal{H}}=1, \ \forall n\geq1, 
\end{equation}
such that 
\begin{equation}\label{pol2-w}
\left(\la_n\right)^{\ell}\left(i\la_nI-\AA\right)U_n=F_n:=(f_{1,n},f_{2,n},f_{3,n},f_{4,n})^{\top}\to 0 \ \ \text{in}\ \ \mathcal{H}, \  \text{as} \ n\to \infty. 
\end{equation}
For simplicity, we drop the index $n$. Equivalently, from \eqref{pol2-w}, we have 
\begin{eqnarray}
i\la u-v&=&\dfrac{f_1}{\la^{\ell}}, \ f_1 \to 0  \ \ \text{in}\ \ H_0^1(0,L),\label{pol3}\\
i\la v-\left(au_x+bv_x\right)_x+cz&=&\dfrac{f_2}{\la^{\ell}}, \ f_2 \to 0  \ \  \text{in}\ \ L^2(0,L),\label{pol4}\\
i\la y-z&=&\dfrac{f_3}{\la^{\ell}}, \ f_3 \to 0  \ \ \text{in}\ \ H_0^1(0,L),\label{pol5}\\
i\la z-(y_x+dz_x)_x-cv&=&\dfrac{f_4}{\la^{\ell}},\  f_4 \to 0  \ \  \text{in} \ \ L^2(0,L).\label{pol6}
\end{eqnarray}
\subsubsection{Proof of Theorem \ref{1pol}} In this subsection, we will prove Theorem \ref{1pol} by checking the condition \eqref{H2},  by finding a contradiction with \eqref{pol1} by showing $\|U\|_{\mathcal{H}}=o(1)$. For clarity, we divide the proof into several Lemmas.
By taking the inner product of \eqref{pol2-w} with $U$ in $\mathcal{H}$, we remark that
\begin{equation*}
\int _0^L b\left|v_{x}\right|^2dx+\int_0^Ld\abs{z_x}^2dx=-\Re\left(\left<\AA U,U\right>_{\HH}\right)=\la^{-\frac{1}{2}}\Re\left(\left<F,U\right>_{\HH}\right)=o\left(\lambda^{-\frac{1}{2}}\right).
\end{equation*}
Thus, from the definitions of $b$ and $d$, we get 
\begin{equation}\label{eq-4.9}
\int _{b_1}^{b_2}\left|v_{x}\right|^2dx=o\left(\lambda^{-\frac{1}{2}}\right)\quad \text{and}\quad \int _{d_1}^{d_2}\left|z_{x}\right|^2dx=o\left(\lambda^{-\frac{1}{2}}\right).
\end{equation}
Using  \eqref{pol3}, \eqref{pol5}, \eqref{eq-4.9}, and the fact that $f_1,f_3\to 0$ in $H_0^1(0,L)$, we get 
\begin{equation}\label{eq-5.0}
\int_{b_1}^{b_2}\abs{u_x}^2dx=\frac{o(1)}{\la^{\frac{5}{2}}}\quad \text{and}\quad \int_{d_1}^{d_2}\abs{y_x}^2dx=\frac{o(1)}{\la^{\frac{5}{2}}}. 
\end{equation}
\begin{lemma}\label{F-est}
	{\rm
The solution $U\in D(\AA)$ of system \eqref{pol3}-\eqref{pol6} satisfies the following estimations
\begin{equation}\label{F-est1}
\int_{b_1}^{b_2}\abs{v}^2dx=\frac{o(1)}{\la^{\frac{3}{2}}}\quad \text{and}\quad \int_{d_1}^{d_2}\abs{z}^2dx=\frac{o(1)}{\la^{\frac{3}{2}}}.
\end{equation}}
\end{lemma}
\begin{proof}
We give the proof of the first estimation in \eqref{F-est1}, the second one can be done in a similar way. For this aim, we fix $g\in C^1\left([b_1,b_2]\right)$ such that 
$$
g(b_2)=-g(b_1)=1,\quad \max_{x\in[b_1,b_2]}\abs{g(x)}=m_g\ \ \text{and}\ \ \max_{x\in [b_1,b_2]}\abs{g'(x)}=m_{g'}.
$$
  The proof is divided into several steps: \\
\textbf{Step 1}. The goal of this step is to prove that 
\begin{equation}\label{Step1-Eq1}
\abs{v(b_1)}^2+\abs{v(b_2)}^2\leq \left(\frac{\la^{\frac{1}{2}}}{2}+2m_{g'}\right)\int_{b_1}^{b_2}\abs{v}^2dx+\frac{o(1)}{\la}.
\end{equation}
From \eqref{pol3}, we deduce that
\begin{equation}\label{Step1-Eq2}
v_x=i\la u_x-\la^{-\frac{1}{2}}(f_1)_x.
\end{equation}
Multiplying \eqref{Step1-Eq2} by $2g\overline{v}$ and integrating over $(b_1,b_2)$, then taking the real part, we get 
\begin{equation*}
\int_{b_1}^{b_2}g\left(\abs{v}^2\right)_xdx=\Re\left(2i\la \int_{b_1}^{b_2}gu_x\overline{v}dx\right)-\Re\left(2\la^{-\frac{1}{2}}\int_{b_1}^{b_2}g(f_1)_x\overline{v}dx\right).
\end{equation*}
Using integration by parts in the left hand side of the above equation, we get 
\begin{equation}\label{Step1-Eq3}
\abs{v(b_1)}^2+\abs{v(b_2)}^2=\int_{b_1}^{b_2}g'\abs{v}^2dx+\Re\left(2i\la \int_{b_1}^{b_2}gu_x\overline{v}dx\right)-\Re\left(2\la^{-\frac{1}{2}}\int_{b_1}^{b_2}g(f_1)_x\overline{v}dx\right).
\end{equation}
Using Young's inequality, we obtain
\begin{equation*}
2\la m_g\abs{u_x}\abs{v}\leq \frac{\la^\frac{1}{2}\abs{v}^2}{2}+2\la^{\frac{3}{2}}m_g^2\abs{u_x}^2\ \text{and}\quad 2\la^{-\frac{1}{2}}m_g\abs{(f_1)_x}\abs{v}\leq m_{g'}\abs{v}^2+m_g^2m_{g'}^{-1}\la^{-1}\abs{(f_1)_x}^2.
\end{equation*}
From the above inequalities, \eqref{Step1-Eq3} becomes 
\begin{equation}\label{Step1-Eq4}
\abs{v(b_1)}^2+\abs{v(b_2)}^2\leq \left(\frac{\la^{\frac{1}{2}}}{2}+2m_{g'}\right)\int_{b_1}^{b_2}\abs{v}^2dx+2\la^{\frac{3}{2}}m_g^2\int_{b_1}^{b_2}\abs{u_x}^2dx+\frac{m_g^2}{m_{g'}}\la^{-1}\int_{b_1}^{b_2}\abs{(f_1)_x}^2dx.
\end{equation}
Inserting \eqref{eq-5.0} in \eqref{Step1-Eq4} and the fact that $f_1 \to 0$ in $H^1_0(0,L)$, we get \eqref{Step1-Eq1}.\\[0.1in] 
\textbf{Step 2}. The aim of this step is to prove that 
\begin{equation}\label{Step2-Eq1}
\abs{(au_x+bv_x)(b_1)}^2+\abs{(au_x+bv_x)(b_2)}^2\leq \frac{\la^{\frac{3}{2}}}{2}\int_{b_1}^{b_2}\abs{v}^2dx+o(1).
\end{equation}
Multiplying \eqref{pol4} by $-2g\left(\overline{au_x+bv_x}\right)$, using integration by parts over $(b_1,b_2)$ and taking the real part, we get 
\begin{equation*}
\begin{array}{l}
\displaystyle
\abs{\left(au_x+bv_x\right)(b_1)}^2+\abs{\left(au_x+bv_x\right)(b_2)}^2=\int_{b_1}^{b_2}g'\abs{au_x+bv_x}^2dx+\\[0.1in]
\displaystyle
\Re\left(2i\la \int_{b_1}^{b_2}gv(\overline{au_x+bv_x})dx\right)-\Re\left(2\la^{-\frac{1}{2}}\int_{b_1}^{b_2}gf_2(\overline{au_x+bv_x})dx\right),
\end{array}
\end{equation*}
consequently, we get
\begin{equation}\label{Step2-Eq2}
\begin{array}{lll}
\displaystyle 
\abs{\left(au_x+bv_x\right)(b_1)}^2+\abs{\left(au_x+bv_x\right)(b_2)}^2\leq m_{g'}\int_{b_1}^{b_2}\abs{au_x+bv_x}^2dx\\[0.1in]
\displaystyle 
+2\la m_g\int_{b_1}^{b_2}\abs{v}\abs{au_x+bv_x}dx+2m_g\la^{-\frac{1}{2}}\int_{b_1}^{b_2}\abs{f_2}\abs{au_x+bv_x}dx.
\end{array}
\end{equation}
By Young's inequality, \eqref{eq-4.9}, and \eqref{eq-5.0}, we have
\begin{equation}\label{Step2-Eq3}
2\la m_g\int_{b_1}^{b_2}\abs{v}\abs{au_x+bv_x}dx\leq \frac{\la^{\frac{3}{2}}}{2}\int_{b_1}^{b_2}\abs{v}^2dx+2m_g^2\la^{\frac{1}{2}}\int_{b_1}^{b_2}\abs{au_x+bv_x}^2dx\leq \frac{\la^{\frac{3}{2}}}{2}\int_{b_1}^{b_2}\abs{v}^2dx+o(1).\\[0.1in]
\end{equation}
Inserting \eqref{Step2-Eq3} in \eqref{Step2-Eq2}, then using  \eqref{eq-4.9}, \eqref{eq-5.0} and the fact that $f_2 \to 0$ in $L^2(0,L)$, we get \eqref{Step2-Eq1}.\\[0.1in]
\textbf{Step 3.} The aim of this step is to prove the first estimation in \eqref{F-est1}. For this aim, multiplying \eqref{pol4} by $-i\la^{-1}\overline{v}$, integrating over $(b_1,b_2)$ and taking the real part , we get
\begin{equation}\label{Step3-Eq1}
\int_{b_1}^{b_2}\abs{v}^2dx=\Re\left(i\la^{-1}\int_{b_1}^{b_2}(au_x+bv_x)\overline{v}_xdx-\left[i\la^{-1}\left(au_x+bv_x\right)\overline{v}\right]_{b_1}^{b_2}+i\la^{-\frac{3}{2}}\int_{b_1}^{b_2}f_2\overline{v}dx\right).
\end{equation}
Using \eqref{eq-4.9}, \eqref{eq-5.0}, the fact that $v$ is uniformly bounded in $L^2(0,L)$ and $f_2\to 0$ in $L^2(0,1)$, and Young's inequalities, we get 
\begin{equation}\label{Step3-Eq2}
\int_{b_1}^{b_2}\abs{v}^2dx\leq \frac{\la^{-\frac{1}{2}}}{2}[\abs{v(b_1)}^2+\abs{v(b_2)}^2]+\frac{\la^{-\frac{3}{2}}}{2}[\abs{(au_x+bv_x)(b_1)}^2+\abs{(au_x+bv_x)(b_2)}^2]+\frac{o(1)}{\la^{\frac{3}{2}}}.
\end{equation}
Inserting \eqref{Step1-Eq1} and \eqref{Step2-Eq1} in \eqref{Step3-Eq2}, we get 
\begin{equation*}
\int_{b_1}^{b_2}\abs{v}^2dx\leq \left(\frac{1}{2}+m_{g'}\la^{-\frac{1}{2}}\right)\int_{b_1}^{b_2}\abs{v}^2dx+\frac{o(1)}{\la^{\frac{3}{2}}}, 
\end{equation*}
which implies that 
\begin{equation}\label{Step3-Eq3}
\left(\frac{1}{2}-m_{g'}\la^{-\frac{1}{2}}\right)\int_{b_1}^{b_2}\abs{v}^2dx\leq \frac{o(1)}{\la^{\frac{3}{2}}}.
\end{equation}
Using the fact that $\la \to \infty$, we can take $\la > 4m_{g'}^2$. Then, we obtain the first estimation in \eqref{F-est1}. Similarly, we can obtain the second estimation in \eqref{F-est1}. The proof has been completed. 
\end{proof}
\begin{lemma}\label{Sec-est}
	{\rm
The solution $U\in D(\AA)$ of system \eqref{pol3}-\eqref{pol6} satisfies the following estimations
\begin{equation}\label{Sec-est1}
\int_0^{c_1}\left(\abs{v}^2+a\abs{u_x}^2\right)dx=o(1)\quad \text{and}\quad \int_{c_2}^L\left(\abs{z}^2+\abs{y_x}^2\right)dx=o(1). 
\end{equation}}
\end{lemma}
\begin{proof}
First, let $h\in C^1([0,c_1])$ such that $h(0)=h(c_1)=0$. Multiplying \eqref{pol4} by $2a^{-1}h\overline{(au_x+bv_x)}$, integrating over $(0,c_1)$, using integration by parts and taking the real part, then using  \eqref{eq-4.9} and the fact that $u_x$ is uniformly bounded in $L^2(0,L)$ and $f_2 \to 0$ in $L^2(0,L)$, we get  
\begin{equation}\label{Sec-est2}
\Re\left(2i\la a^{-1}\int_0^{c_1}vh\overline{(au_x+bv_x)}dx\right)+a^{-1}\int_0^{c_1}h'\abs{au_x+bv_x}^2dx=\frac{o(1)}{\la^{\frac{1}{2}}}.
\end{equation}
From \eqref{pol3}, we have 
\begin{equation}\label{Sec-est3}
i\la \overline{u}_x=-\overline{v}_x-\la^{-\frac{1}{2}}(\overline{f_1})_x.
\end{equation}
Inserting \eqref{Sec-est3} in \eqref{Sec-est2}, using integration by parts, then using \eqref{eq-4.9}, \eqref{F-est1}, and the fact that $f_1 \to 0 $ in $H^1_0 (0,L)$ and $v$ is uniformly bounded in $L^2 (0,L)$, we get 
\begin{equation}\label{Sec-est4}
\begin{array}{c}
\displaystyle
\int_0^{c_1}h'\abs{v}^2dx+a^{-1}\int_0^{c_1}h'\abs{au_x+bv_x}^2dx=\underbrace{2\Re\left(\la^{-\frac{1}{2}}\int_{0}^{c_1}vh(\overline{f_1})_xdx\right)}_{=o(\la^{-\frac{1}{2}})}\\[0.1in]
\displaystyle 
+\underbrace{\Re\left(2i\la a^{-1}b_0\int_{b_1}^{b_2}hv\overline{v}_xdx\right)}_{=o(1)}+\frac{o(1)}{\la^{\frac{1}{2}}}.
\end{array}
\end{equation}
Now, we fix the following cut-off functions
$$
p_1(x):=\left\{\begin{array}{ccc}
1&\text{in}&(0,b_1),\\
0&\text{in}&(b_2,c_1),\\
0\leq p_1\leq 1&\text{in}&(b_1,b_2),
\end{array}
\right. \quad\text{and}\quad 
p_2(x):=\left\{\begin{array}{ccc}
1&\text{in}&(b_2,c_1),\\
0&\text{in}&(0,b_1),\\
0\leq p_2\leq 1&\text{in}&(b_1,b_2).
\end{array}
\right.
$$
Finally, take $h(x)=xp_1(x)+(x-c_1)p_2(x)$ in \eqref{Sec-est4} and using \eqref{eq-4.9}, \eqref{eq-5.0}, \eqref{F-est1}, we get the first estimation in  \eqref{Sec-est1}. By using the same argument, we can obtain the second estimation in \eqref{Sec-est1}. The proof is thus completed. 
\end{proof}
\begin{lemma}\label{Third-est}
The solution $U\in D(\AA)$ of system \eqref{pol3}-\eqref{pol6} satisfies the following estimations
\begin{equation}\label{Third-est1}
\abs{\la u(c_1)}=o(1),\ \abs{u_x(c_1)}=o(1),\ \abs{\la y(c_2)}=o(1)\quad \text{and}\quad \abs{y_x(c_2)}=o(1).
\end{equation}
\end{lemma}
\begin{proof}
First, from \eqref{pol3} and \eqref{pol4}, we deduce that  
\begin{equation}\label{Th-est1}
\la^2u+au_{xx}=-\frac{f_2}{\la^{\frac{1}{2}}}-i\la^{\frac{1}{2}}f_1 \ \ \text{in} \ \ (b_2,c_1). 
\end{equation}
Multiplying \eqref{Th-est1} by $2(x-b_2)\bar{u}_x$, integrating over $(b_2,c_1)$ and taking the real part, then using the fact that $u_x$ is uniformly bounded in $L^2(0,L)$ and $f_2 \to 0$ in $L^2(0,L)$, we get 
\begin{equation}\label{Th-est2}
\int_{b_2}^{c_1}\la^2 (x-b_2)\left(\abs{u}^2\right)_xdx+a\int_{b_2}^{c_1}(x-b_2)\left(\abs{u_x}^2\right)_xdx=-\Re\left(2i\la^{\frac{1}{2}}\int_{b_2}^{c_1}(x-b_2)f_1\overline{u}_xdx\right)+\frac{o(1)}{\la^{\frac{1}{2}}}.
\end{equation}
Using integration by parts in \eqref{Th-est2}, then using \eqref{Sec-est1}, and the fact that $f_1\to 0$ in $H_0^1(0,L)$ and $\la u$ is uniformly bounded in $L^2(0,L)$, we get 
\begin{equation}\label{Th-est3}
0\leq (c_1-b_2)\left(\abs{\la u(c_1)}^2+a\abs{u_x(c_1)}^2\right)=\Re\left(2i\la^{\frac{1}{2}}(c_1-b_2)f_1(c_1)\overline{u}(c_1)\right)+o(1),
\end{equation}
consequently, by using Young's inequality, we get
\begin{equation*}
\begin{array}{lll}
\displaystyle\abs{\la u(c_1)}^2+\abs{u_x(c_1)}^2 &\leq& \displaystyle 2\la^{\frac{1}{2}}|f_1(c_1)||u(c_1)|+o(1)\\[0.1in]
&\leq &\displaystyle\frac{1}{2}\abs{\la u(c_1)}^2+\frac{2}{\la}\abs{f_1(c_1)}^2 +o(1).
\end{array}
\end{equation*}
Then, we get 
\begin{equation}
\frac{1}{2}\abs{\la u(c_1)}^2+\abs{u_x(c_1)}^2\leq \frac{2}{\la}\abs{f_1(c_1)}^2+o(1). 
\end{equation}
Finally, from the above estimation and the fact that $f_1 \to 0$ in $H^1_0 (0,L)$, we get the first two estimations in \eqref{Third-est1}. By using the same argument, we can obtain the last two estimations in \eqref{Third-est1}. The proof has been completed. 
\end{proof}
\begin{lemma}\label{Fourth-est}
The solution $U\in D(\AA)$ of system \eqref{pol3}-\eqref{pol6} satisfies the following estimation
\begin{equation}\label{4-est1}
\int_{c_1}^{c_2} |\la u|^2 +a |u_x|^2 +|\la y|^2 +|y_x|^2 dx =o(1).
\end{equation}
\end{lemma}
\begin{proof}
Inserting \eqref{pol3} and \eqref{pol5} in \eqref{pol4} and \eqref{pol6}, we get 
\begin{eqnarray}
-\la^2u-au_{xx}+i\la c_0y&=&\frac{f_2}{\la^{\frac{1}{2}}}+i\la^{\frac{1}{2}}f_1+\frac{c_0f_3}{\la^{\frac{1}{2}}} \ \ \text{in} \ \ (c_1,c_2),\label{4-est2}\\
-\la^2y-y_{xx}-i\la c_0u&=&\frac{f_4}{\la^{\frac{1}{2}}}+i\la^{\frac{1}{2}}f_3-\frac{c_0f_1}{\la^{\frac{1}{2}}} \ \ \ \text{in} \ \ (c_1,c_2)\label{4-est3}. 
\end{eqnarray}
Multiplying \eqref{4-est2} by $2(x-c_2)\overline{u_x}$ and \eqref{4-est3} by $2(x-c_1)\overline{y_x}$, integrating over $(c_1,c_2)$ and taking the real part, then using the fact that $\|F\|_\HH =o(1)$ and $\|U\|_\HH =1$, we obtain
\begin{equation}\label{4-est4}
\begin{array}{l}
\displaystyle 
-\la^2\int_{c_1}^{c_2}(x-c_2)\left(\abs{u}^2\right)_xdx-a\int_{c_1}^{c_2}(x-c_2)\left(\abs{u_x}^2\right)_xdx+\Re\left(2i\la c_0\int_{c_1}^{c_2}(x-c_2)y\overline{u_x}dx\right)=\vspace{0.25cm}\\
\displaystyle
\Re\left(2i\la^{\frac{1}{2}}\int_{c_1}^{c_2}(x-c_2)f_1\overline{u_x}dx\right)+\frac{o(1)}{\la^{\frac{1}{2}}}
\end{array}
\end{equation}
and
\begin{equation}\label{4-est5}
\begin{array}{l}
\displaystyle 
-\la^2\int_{c_1}^{c_2}(x-c_1)\left(\abs{y}^2\right)_xdx-\int_{c_1}^{c_2}(x-c_1)\left(\abs{y_x}^2\right)_xdx-\Re\left(2i\la c_0\int_{c_1}^{c_2}(x-c_1)u\overline{y_x}dx\right)=\vspace{0.25cm}\\
\displaystyle
\Re\left(2i\la^{\frac{1}{2}}\int_{c_1}^{c_2}(x-c_1)f_3\overline{y_x}dx\right)+\frac{o(1)}{\la^{\frac{1}{2}}}.
\end{array}
\end{equation}
Using integration by parts, \eqref{Third-est1}, and the fact that $f_1, f_3 \to 0$ in $H^1_0(0,L)$,  $\|u\|_{L^2(0,L)}=O(\la^{-1})$, $\|y\|_{L^2(0,L)}=O(\la^{-1})$, we deduce that
\begin{equation}\label{4-est6}
\Re\left(i\la^{\frac{1}{2}}\int_{c_1}^{c_2}(x-c_2)f_1\overline{u_x}dx\right)=\frac{o(1)}{\la^{\frac{1}{2}}}\quad \text{and}\quad \Re\left(i\la^{\frac{1}{2}}\int_{c_1}^{c_2}(x-c_1)f_3\overline{y_x}dx\right)=\frac{o(1)}{\la^{\frac{1}{2}}}.
\end{equation}
Inserting \eqref{4-est6} in \eqref{4-est4} and \eqref{4-est5}, then using integration by parts and \eqref{Third-est1}, we get 
\begin{eqnarray}
\int_{c_1}^{c_2}\left(\abs{\la u}^2+a\abs{u_x}^2\right)dx+\Re\left(i\la c_0\int_{c_1}^{c_2}(x-c_2)y\overline{u_x}dx\right)&=&o(1),\label{4-est7}\\
\int_{c_1}^{c_2}\left(\abs{\la y}^2+\abs{y_x}^2\right)dx-\Re\left(i\la c_0\int_{c_1}^{c_2}(x-c_1)u\overline{y_x}dx\right)&=&o(1).\label{4-est8}
\end{eqnarray}
Adding \eqref{4-est7} and \eqref{4-est8}, we get 
$$
\begin{array}{lll}
\displaystyle
\int_{c_1}^{c_2}\left(\abs{\la u}^2+a\abs{u_x}^2+\abs{\la y}^2+\abs{y_x}^2\right)dx&=&\displaystyle 
\Re\left(2i\la c_0\int_{c_1}^{c_2}(x-c_1)u\overline{y_x}dx\right)-\Re\left(2i\la c_0\int_{c_1}^{c_2}(x-c_2)y\overline{u_x}dx\right)+o(1)\\[0.in]
&\leq &\displaystyle 
2\la \abs{c_0}(c_2-c_1)\int_{c_1}^{c_2}\abs{u}\abs{y_x}dx+2\la\frac{\abs{c_0}}{a^{\frac{1}{4}}}(c_2-c_1)a^{\frac{1}{4}}\int_{c_1}^{c_2}\abs{y}\abs{u_x}dx+o(1).
\end{array}
$$
Applying Young's inequalities, we get 
\begin{equation}\label{4-est9}
\left(1-\abs{c_0}(c_2-c_1)\right)\int_{c_1}^{c_2}(\abs{\la u}^2+\abs{y_x}^2)dx+\left(1-\frac{1}{\sqrt{a}}\abs{c_0}(c_2-c_1)\right)\int_{c_1}^{c_2}(a\abs{u_x}^2+\abs{\la y}^2)dx\leq o(1).
\end{equation}
Finally, using \eqref{SSC1}, we get the desired result. The proof has been completed. 
\end{proof}
\begin{lemma}\label{5-est}
The solution $U\in D(\AA)$ of system \eqref{pol3}-\eqref{pol6} satisfies the following estimations
\begin{equation}\label{5-est1}
\int_0^{c_1}\left(\abs{z}^2+\abs{y_x}^2\right)dx=o(1)\quad \text{and}\quad \int_{c_2}^L\left(\abs{v}^2+a\abs{u_x}^2\right)dx=o(1). 
\end{equation}
\end{lemma}
\begin{proof}
Using the same argument of Lemma \ref{Sec-est}, we obtain \eqref{5-est1}. 
\end{proof}

\noindent \textbf{Proof of Theorem \ref{1pol}.} Using \eqref{eq-5.0}, Lemmas \ref{F-est}, \ref{Sec-est},  \ref{Fourth-est}, \ref{5-est}, we get $\|U\|_{\mathcal{H}}=o(1)$, which contradicts \eqref{pol1}. Consequently, condition ${\rm (H2)}$ holds. This implies the energy decay estimation \eqref{Energypol1}.


\subsubsection{Proof of Theorem \ref{2pol}} In this subsection, we will prove Theorem \ref{2pol} by checking the condition \eqref{H2}, that is by finding a contradiction with \eqref{pol1} by showing $\|U\|_{\mathcal{H}}=o(1)$. For clarity, we divide the proof into several Lemmas. By taking the inner product of \eqref{pol2-w} with $U$ in $\mathcal{H}$, we remark that 
\begin{equation*}
\int_0^L b\abs{v_x}^2dx=-\Re\left(\left<\mathcal{A}U,U\right>_{\mathcal{H}}\right)=\la^{-2}\Re\left(\left<F,U\right>_{\mathcal{H}}\right)=o(\la^{-2}).
\end{equation*}
Then, 
\begin{equation}\label{C2-dissipation}
\int_{b_1}^{b_2}\abs{v_x}^2dx=o(\la^{-2}).
\end{equation}
Using \eqref{pol3} and \eqref{C2-dissipation}, and the fact that $f_1 \to 0$ in $H^1_0(0,L)$, we get 
\begin{equation}\label{C2-dissipation1}
\int_{b_1}^{b_2}\abs{u_x}^2dx=o(\la^{-4}).
\end{equation}
\begin{lemma}\label{C2-Fest}
Let $0<\varepsilon<\frac{b_2-b_1}{2}$, the solution $U\in D(\mathcal{A})$ of the system \eqref{pol3}-\eqref{pol6} satisfies the following estimation 
\begin{equation}\label{C2-Fest1}
\int_{b_1+\varepsilon}^{b_2-\varepsilon}\abs{v}^2dx=o(\la^{-2}). 
\end{equation}
\end{lemma}
\begin{proof}
First, we fix a cut-off function $\theta_1\in C^{1}([0,c_1])$ such that 
\begin{equation}\label{C2-theta1}
\theta_1(x)=\left\{\begin{array}{clc}
1&\text{if}&x\in (b_1+\varepsilon,b_2-\varepsilon),\\
0&\text{if}&x\in (0,b_1)\cup (b_2,L),\\
0\leq \theta_1\leq 1&&\text{elsewhere}.
\end{array}
\right.
\end{equation}
Multiplying \eqref{pol4} by $\la^{-1}\theta_1 \overline{v}$, integrating over $(0,c_1)$, using integration by parts, and the fact that $f_2 \to 0$ in $L^2(0,L)$ and $v$ is uniformly bounded in $L^2(0,L)$, we get 
\begin{equation}\label{C2-Fest2}
i\int_0^{c_1}\theta_1\abs{v}^2dx+\frac{1}{\la}\int_0^{c_1}(u_x+bv_x)(\theta_1'\overline{v}+\theta \overline{v_x})dx=o(\la^{-3}).
\end{equation}
Using \eqref{C2-dissipation} and the fact that $\|U\|_{\mathcal{H}}=1$, we get 
\begin{equation*}
\frac{1}{\la}\int_0^{c_1}(u_x+bv_x)(\theta_1'\overline{v}+\theta \overline{v_x})dx=o(\la^{-2}).
\end{equation*}
Inserting the above estimation in \eqref{C2-Fest2}, we get the desired result \eqref{C2-Fest1}. The proof has been completed. 
\end{proof}
\begin{lemma}\label{C2-Secest}
The solution $U\in D(\mathcal{A})$ of the system \eqref{pol3}-\eqref{pol6} satisfies the following estimation 
\begin{equation}\label{C2-Secest1}
\int_{0}^{c_1}(\abs{v}^2+\abs{u_x}^2)dx=o(1). 
\end{equation}
\end{lemma}
\begin{proof}
Let $h\in C^1([0,c_1])$ such that $h(0)=h(c_1)=0$. Multiplying \eqref{pol4} by $2h\overline{(u_x+bv_x)}$, integrating over $(0,c_1)$ and taking the real part, then using integration by parts and the fact that $f_2 \to 0$ in $L^2(0,L)$, we get 
\begin{equation}\label{C2-Secest2}
	\Re\left(2\int_0^{c_1}i\la vh\overline{(u_x+bv_x)}dx\right)+\int_0^{c_1}h'\abs{u_x+bv_x}^2dx=o(\la^{-2}).
\end{equation}
Using \eqref{C2-dissipation} and the fact that $v$ is uniformly bounded in $L^2(0,L)$, we get 
\begin{equation}\label{C2-Secest3}
\Re\left(2\int_0^{c_1}i\la vh\overline{(u_x+bv_x)}dx\right)=2\int_0^{c_1}i\la vh\overline{u_x}dx+o(1). 
\end{equation}
From \eqref{pol3}, we have
\begin{equation}\label{C2-Secest4}
i\la\overline{u}_x=-\overline{v}_x-\frac{\left(\overline{f_1}\right)_x}{\la^2}.
\end{equation}
Inserting \eqref{C2-Secest4} in \eqref{C2-Secest3}, using integration  by parts and the fact that $f_1 \to 0$ in $H^1_0(0,L)$, we get 
\begin{equation}\label{C2-Secest5}
\Re\left(2\int_0^{c_1}i\la vh\overline{(u_x+bv_x)}dx\right)=\int_0^{c_1}h'\abs{v}^2dx+o(1).
\end{equation}
Inserting \eqref{C2-Secest5} in \eqref{C2-Secest2}, we obtain
\begin{equation}\label{C2-Secest6}
\int_0^{c_1}h'\left(\abs{v}^2+\abs{u_x+bv_x}^2\right)dx=o(1).
\end{equation}
Now, we fix the following cut-off functions
$$
\theta_2(x):=\left\{\begin{array}{ccc}
1&\text{in}&(0,b_1+\varepsilon),\\
0&\text{in}&(b_2-\varepsilon,c_1),\\
0\leq \theta_2\leq 1&\text{in}&(b_1+\varepsilon,b_2-\varepsilon),
\end{array}
\right. \quad\text{and}\quad 
\theta_3(x):=\left\{\begin{array}{ccc}
1&\text{in}&(b_2-\varepsilon,c_1),\\
0&\text{in}&(0,b_1+\varepsilon),\\
0\leq \theta_3\leq 1&\text{in}&(b_1+\varepsilon,b_2-\varepsilon).
\end{array}
\right.
$$
Taking $h(x)=x\theta_2(x)+(x-c_1)\theta_3(x)$ in \eqref{C2-Secest6}, then using  \eqref{C2-dissipation} and  \eqref{C2-dissipation1}, we get 
\begin{equation}\label{C2-Secest7}
\int_{(0,b_1+\varepsilon)\cup (b_2-\varepsilon,c_1)}\abs{v}^2dx+\int_{(0,b_1)\cup (b_2,c_1)}|u_x|^2dx=o(1).
\end{equation}
Finally, from  \eqref{C2-dissipation1}, \eqref{C2-Fest1} and \eqref{C2-Secest7},  we get the desired result \eqref{C2-Secest1}. The proof has been completed. 
\end{proof}

\noindent  
\begin{lemma}\label{C2-Fourthest}
The solution $U\in D(\AA)$ of system \eqref{pol3}-\eqref{pol6} satisfies the following estimations
\begin{equation}\label{C2-Thirest1}
	\abs{\la u(c_1)}=o(1)\quad \text{and}\quad \abs{u_x(c_1)}=o(1),
\end{equation}
\begin{equation}\label{C2-Fourthest1}
\int_{c_1}^{c_2}\abs{\la u}^2dx=\int_{c_1}^{c_2}\abs{\la y}^2dx+o(1).
\end{equation}
\end{lemma}
\begin{proof}
	First, using the same argument of Lemma \ref{Third-est}, we claim \eqref{C2-Thirest1}. 
 Inserting \eqref{pol3}, \eqref{pol5} in \eqref{pol4} and \eqref{pol6}, we get 
	\begin{eqnarray}
		\la^2u+\left(u_x+bv_x\right)_x-i\la cy&=&-\frac{f_2}{\la^{2}}-i\frac{f_1}{\la}-c\frac{f_3}{\la^2},\label{Combination1}\\
		\la^2y+y_{xx}+i\la cu&=&-\frac{f_4}{\la^2}-\frac{if_3}{\la}+c\frac{f_1}{\la^2}.\label{Combination2}
	\end{eqnarray}
Multiplying \eqref{Combination1} and \eqref{Combination2} by $\la \overline{y}$ and $\la \overline{u}$ respectively, integrating over $(0,L)$, then using integration by parts, \eqref{C2-dissipation}, and the fact that $\|U\|_\HH=1$ and $\|F\|_\HH =o(1)$, we get 
\begin{eqnarray}
\la^{3}\int_0^Lu\bar{y}dx-\la\int_0^Lu_x\bar{y}_xdx-i c_0\int_{c_1}^{c_2}\abs{\la y}^2dx=o(1),\label{C2-Fourthest2}\\
\la^{3}\int_0^Ly\bar{u}dx-\la \int_0^Ly_x\bar{u}_xdx+i c_0\int_{c_1}^{c_2}\abs{\la u}^2dx=\frac{o(1)}{\la}\label{C2-Fourthest3}.
\end{eqnarray}
Adding \eqref{C2-Fourthest2} and \eqref{C2-Fourthest3} and taking the imaginary parts, we get the desired result \eqref{C2-Fourthest1}. The proof is thus completed. 
\end{proof}
\begin{lemma}\label{C2-Fifthest}
The solution $U\in D(\AA)$ of system \eqref{pol3}-\eqref{pol6} satisfies the following asymptotic behavior
\begin{equation}\label{C2-Fifthest1}
\int_{c_1}^{c_2}\abs{\la u}^2dx=o(1),\quad \int_{c_1}^{c_2}\abs{\la y}^2dx=o(1)\quad \text{and}\quad \int_{c_1}^{c_2}\abs{u_x}^2dx=o(1).
\end{equation}
\end{lemma}
\begin{proof}
First, Multiplying \eqref{Combination1} by $2(x-c_2)\bar{u}_x$, integrating over $(c_1,c_2)$ and taking the real part, using the fact that $\|U\|_\HH=1$ and $\|F\|_\HH =o(1)$, we get
\begin{equation}\label{C2-Fifthest2}
\la^2\int_{c_1}^{c_2}(x-c_2)\left(\abs{u}^2\right)_xdx+\int_{c_1}^{c_2}(x-c_2)\left(\abs{u_x}^2\right)_xdx=\Re\left(2i\la c_0\int_{c_1}^{c_2}(x-c_2)y\bar{u}_xdx\right)+o(1).
\end{equation}
Using integration by parts in \eqref{C2-Fifthest2} with the help of \eqref{C2-Thirest1}, we get 
\begin{equation}\label{C2-Fifthest3}
\int_{c_1}^{c_2}\abs{\la u}^2dx+\int_{c_1}^{c_2}\abs{u_x}^2dx\leq 2\la\abs{c_0}(c_2-c_1)\int_{c_1}^{c_2}\abs{y}\abs{u_x}+o(1).
\end{equation}
Applying Young's inequality in \eqref{C2-Fifthest3}, we get 
\begin{equation}\label{C2-Fifthest4}
\int_{c_1}^{c_2}\abs{\la u}^2dx+\int_{c_1}^{c_2}\abs{u_x}^2dx\leq \abs{c_0}(c_2-c_1)\int_{c_1}^{c_2}\abs{u_x}^2dx+\abs{c_0}(c_2-c_1)\int_{c_1}^{c_2}\abs{\la y}^2dx+o(1). 
\end{equation}
Using \eqref{C2-Fourthest1} in \eqref{C2-Fifthest4}, we get 
\begin{equation}\label{C2-Fifthest5}
\left(1-\abs{c_0}(c_2-c_1)\right)\int_{c_1}^{c_2}\left(\abs{\la u}^2+|u_x|^2\right)dx\leq o(1). 
\end{equation}
Finally, from the above estimation, \eqref{SSC2} and  \eqref{C2-Fourthest1}, we get the desired result \eqref{C2-Fifthest1}. The proof has been completed. 
\end{proof}
\begin{lemma}\label{C2-sixthest}
Let $0<\delta<\frac{c_2-c_1}{2}$. The solution $U\in D(\AA)$ of system \eqref{pol3}-\eqref{pol6} satisfies the following estimations
\begin{equation}\label{C2-sixthest1}
\int_{c_1+\delta}^{c_2-\delta}\abs{y_x}^2dx=o(1). 
\end{equation}
\end{lemma}
\begin{proof}
First, we fix a  cut-off function $\theta_4\in C^1([0,L])$ such that 
\begin{equation}\label{C2-theta4}
\theta_4(x):=\left\{\begin{array}{clc}
1&\text{if}&x\in (c_1+\delta,c_2-\delta),\\
0&\text{if}&x\in (0,c_1)\cup (c_2,L),\\
0\leq \theta_4\leq 1&&\text{elsewhere}.
\end{array}
\right.
\end{equation}
Multiplying \eqref{Combination2} by $\theta_4\bar{y}$, integrating over $(0,L)$ and using integration by parts, we get 
\begin{equation}\label{C2-sixthest2*}
\int_{c_1}^{c_2}\theta_4\abs{\la y}^2dx-\int_{0}^{L}\theta_4\abs{y_x}^2dx-\int_0^L\theta_4'y_x\bar{y}dx+i\la c_0\int_{c_1}^{c_2}\theta_4u\bar{y}dx=\frac{o(1)}{\la^2}. 
\end{equation}
Using  \eqref{C2-Fifthest1} and the definition of $\theta_4$, we get 
\begin{equation}\label{C2-sixthest3}
\int_{c_1}^{c_2}\theta_4\abs{\la y}^2dx=o(1),\quad \int_0^L\theta_4'y_x\bar{y}dx=o(\la^{-1}),\quad i\la c_0\int_{c_1}^{c_2}\theta_4u\bar{y}dx=o(\la^{-1}).
\end{equation}
Finally, Inserting \eqref{C2-sixthest3} in \eqref{C2-sixthest2*}, we get the desired result \eqref{C2-sixthest1}. The proof has been completed. 
\end{proof}
\begin{lemma}\label{C2-seventhest}
The solution $U\in D(\AA)$ of system \eqref{pol3}-\eqref{pol6} satisfies the following estimations
\begin{equation}\label{C2-sixthest1}
\int_0^{c_1+\varepsilon}\abs{\la y}^2dx,\int_{0}^{c_1+\varepsilon}\abs{y_x}^2dx,\int_{c_2-\varepsilon}^L\abs{\la y}^2dx,\int_{c_2-\varepsilon}^L\abs{y_x}^2dx,\int_{c_2}^{L}\abs{\la u}^2dx,\int_{c_2}^{L}\abs{u_x}^2dx=o(1).
\end{equation}
\end{lemma}
\begin{proof}
Let $q\in C^1([0,L])$ such that $q(0)=q(L)=0$. Multiplying \eqref{Combination1} by $2q\bar{y}_x$ integrating over $(0,L)$, using \eqref{C2-Fifthest1}, and the fact that $y_x$ is uniformly bounded in $L^2(0,L)$ and $\|F\|_{\mathcal{H}}=o(1)$, we get 
\begin{equation}\label{C2-sixthest2}
\int_0^{L}q'\left(\abs{\la y}^2+\abs{y_x}^2\right)dx=o(1).
\end{equation}
Now, take $q(x)=x\theta_5(x)+(x-L)\theta_6(x)$ in \eqref{C2-sixthest2}, such that 
$$
\theta_5(x):=\left\{\begin{array}{ccc}
1&\text{in}&(0,c_1+\varepsilon),\\
0&\text{in}&(c_2-\varepsilon,L),\\
0\leq \theta_1\leq 1&\text{in}&(c_1+\varepsilon,c_2-\varepsilon),
\end{array}
\right. \quad\text{and}\quad 
\theta_2(x)\left\{\begin{array}{ccc}
1&\text{in}&(c_2-\varepsilon,L),\\
0&\text{in}&(0,c_1+\varepsilon),\\
0\leq \theta_2\leq 1&\text{in}&(c_1+\varepsilon,c_2-\varepsilon).
\end{array}
\right.
$$
Then, we obtain the first four estimations in \eqref{C2-sixthest1}. Now, multiplying \eqref{Combination1} by $2q\left(\overline{u_x+bv_x}\right)$ integrating over $(0,L)$ and using the fact that $u_x$ is uniformly bounded in $L^2(0,L)$, we get 
\begin{equation}
\int_0^Lq'\left(\abs{\la u}^2+\abs{u_x}^2\right)dx=o(1). 
\end{equation}
By taking $q(x)=(x-L)\theta_7(x)$, such that 
$$
\theta_7(x)=\left\{\begin{array}{ccc}
1&\text{in}&(c_2,L),\\
0&\text{in}&(0,c_1),\\
0\leq \theta_7\leq 1&\text{in}&(c_1,c_2),
\end{array}
\right.
$$
 we get the the last two estimations in \eqref{C2-sixthest1}. The proof has been completed. 
\end{proof}

\noindent \textbf{Proof of Theorem \ref{2pol}.} Using \eqref{C2-dissipation1}, Lemmas \ref{C2-Secest}, \ref{C2-Fifthest}, \ref{C2-sixthest} and \ref{C2-seventhest}, we get $\|U\|_{\mathcal{H}}=o(1)$,   which contradicts \eqref{pol1}. Consequently, condition ${\rm (H2)}$ holds. This implies the energy decay estimation \eqref{Energypol2}

\section{Indirect Stability in the multi-dimensional case}\label{secnd}
\noindent In this section, we study the well-posedness and the strong stability of system \eqref{ND-1}-\eqref{ND-5}.
\subsection{Well-posedness}\label{wpnd} In this subsection, we will establish the well-posedness of \eqref{ND-1}-\eqref{ND-5} by usinf semigroup approach. The energy of system \eqref{ND-1}-\eqref{ND-5} is given by
\begin{equation}\label{ND-energy}
E(t)=\frac{1}{2}\int_0^L\left(\abs{u_t}^2+\abs{\nabla u}^2+\abs{y_t}^2+\abs{\nabla y}^2\right)dx.
\end{equation}
Let $(u,u_t,y,y_t)$ be a regular solution of \eqref{ND-1}-\eqref{ND-5}. Multiplying \eqref{ND-1} and \eqref{ND-2} by $\overline{u_t}$ and $\overline{y_t}$ respectively, then using the boundary conditions \eqref{ND-3}, we get 
\begin{equation}\label{ND-denergy}
E'(t)=-\int_{\Omega}b|\nabla u_{t}|^2dx,
\end{equation}
using the definition of $b$, we get $E'(t)\leq 0$. Thus, system \eqref{ND-1}-\eqref{ND-5} is dissipative in the sense that its energy is non-increasing with respect to time $t$. Let us define the energy space $\mathcal{H}$ by 
$$
\mathcal{H}=\left(H_0^1(\Omega)\times L^2(\Omega)\right)^2.
$$
The energy space $\mathcal{H}$ is equipped with the inner product defined by 
$$
\left<U,U_1\right>_{\mathcal{H}}=\int_{\Omega}v\overline{v_1}dx+\int_{\Omega}\nabla{u}\nabla{\overline{u_1}}dx+\int_{\Omega}z\overline{z_1}dx+\int_{\Omega}\nabla{y}\cdot \nabla{\overline{y_1}}dx,
$$
for all $U=(u,v,y,z)^\top$ and $U_1=(u_1,v_1,y_1,z_1)^\top$ in $\mathcal{H}$. We define the unbounded linear operator $\mathcal{A}_d:D\left(\mathcal{A}_d\right)\subset \mathcal{H}\longrightarrow \mathcal{H}$ by 
$$
D(\mathcal{A}_d)=\left\{
U=(u,v,y,z)^\top\in \mathcal{H};\ v,z\in H_0^1(\Omega),\  \divv(u_x+bv_x)\in L^2(\Omega),\ \Delta y \in L^2 (\Omega)
\right\}
$$
and 
$$
\mathcal{A}_d U=\begin{pmatrix}
v\\[0.1in] \divv(\nabla u+b\nabla v)-cz\\[0.1in] z\\ \Delta y+cv
\end{pmatrix}, \ \forall U=(u,v,y,z)^\top \in D(\mathcal{A}_d).
$$
If $U=(u,u_t,y,y_t)$ is a regular solution of system \eqref{ND-1}-\eqref{ND-5}, then we rewrite this system as the following first order evolution equation 
\begin{equation}\label{ND-evolution}
U_t=\mathcal{A}_dU,\quad U(0)=U_0, 
\end{equation}
where $U_0=(u_0,u_1,y_0,y_1)^{\top}\in \HH$. For all $U=(u,v,y,z)^{\top}\in D(\mathcal{A}_d )$, we have 
$$
\Re\left<\mathcal{A}_d U,U\right>_{\mathcal{H}}=-\int_{\Omega}b\abs{\nabla v}^2dx\leq 0,
$$
which implies that $\mathcal{A}_d$ is dissipative. Now, similar to Proposition 2.1 in \cite{akil2021ndimensional}, we can prove that there exists a unique solution $U=(u,v,y,z)^{\top}\in D(\mathcal{A}_d)$ of 
$$
-\AA_d U=F,\quad \forall F=(f^1,f^2,f^3,f^4)^\top\in \mathcal{H}.
$$
Then $0\in \rho(\mathcal{A}_d)$ and $\mathcal{A}_d$ is an isomorphism and since $\rho(\mathcal{A}_d)$ is open in $\mathbb{C}$ (see Theorem 6.7 (Chapter III) in \cite{Kato01}),  we easily get $R(\lambda I -\mathcal{A}_d) = {\mathcal{H}}$ for a sufficiently small $\lambda>0 $. This, together with the dissipativeness of $\mathcal{A}_d$, imply that   $D\left(\mathcal{A}_d\right)$ is dense in ${\mathcal{H}}$   and that $\mathcal{A}_d$ is m-dissipative in ${\mathcal{H}}$ (see Theorems 4.5, 4.6 in  \cite{Pazy01}).
According to Lumer-Phillips theorem (see \cite{Pazy01}), then the operator $\AA_d$ generates a $C_{0}$-semigroup of contractions $e^{t\AA_d}$ in $\HH$ which gives the well-posedness of \eqref{ND-evolution}. Then, we have the following result:
\begin{theoreme}{\rm
For all $U_0 \in \HH$,  system \eqref{eq-2.9} admits a unique weak solution $$U(t)=e^{t\AA_d}U_0\in C^0 (\R_+ ,\HH).
	$$ Moreover, if $U_0 \in D(\AA)$, then the system \eqref{eq-2.9} admits a unique strong solution $$U(t)=e^{t\AA_d}U_0\in C^0 (\R_+ ,D(\AA_d))\cap C^1 (\R_+ ,\HH).$$}
\end{theoreme}

\subsection{Strong Stability }\label{Strong Stability-ND}
In this subsection, we will prove the strong stability of system \eqref{ND-1}-\eqref{ND-5}. First, we fix the following notations
$$
\widetilde{\Omega}=\Omega-\overline{\omega_c},\quad \Gamma_1=\partial \omega_c-\partial \Omega\quad \text{and}\quad \Gamma_0=\partial\omega_c-\Gamma_1. 
$$ 
\begin{figure}
\begin{center}
\begin{tikzpicture}
\draw[black] (0,0) (45:4cm) arc (45:315:4cm);
\draw[color=blue] (-1.5,0) circle [radius=1];
\draw[red] (0,0) (-45:4cm)   arc (-45:45:4cm);
\draw [purple] (45:4)   arc (-45:45:-4);
\node[black,left] at (1,2.65){\scalebox{1.5}{$\widetilde{\Omega}$}};
\node[black,left] at (0,-4.65){\scalebox{1.5}{$\Omega$}};
\node[red,left] at (3,0){\scalebox{1.5}{$\omega_c$}};
\node[blue,left] at (-1,0){\scalebox{1.5}{$\omega_b$}};
\node[black,left] at (6,0){\scalebox{1}{$\bullet$}};
\draw[dashed] (2.83,2.83) -- (5.77,0);
\draw[dashed] (2.83,-2.83) -- (5.77,0);
\node[black] at (6.25,0) {\scalebox{1.25}{$x_0$}};
\node[red] at (4.25,-0.75) {\scalebox{1.25}{$\Gamma_0$}};
\node[purple,left] at (1.7,0) {\scalebox{1.25}{$\Gamma_1$}};
\draw[->] (4,0.02)--(5,0.02);
\node[black] at (4,0){\scalebox{1}{$\bullet$}};
\node[black] at (3.7,0){\scalebox{1.25}{$x$}};
\node[black] at (4.35,0.25) {\scalebox{1.25}{$\nu$}};
\end{tikzpicture}
\end{center}
\caption{Geometric description of the sets $\omega_b$ and $\omega_c$}\label{p7-Fig4}
\end{figure}
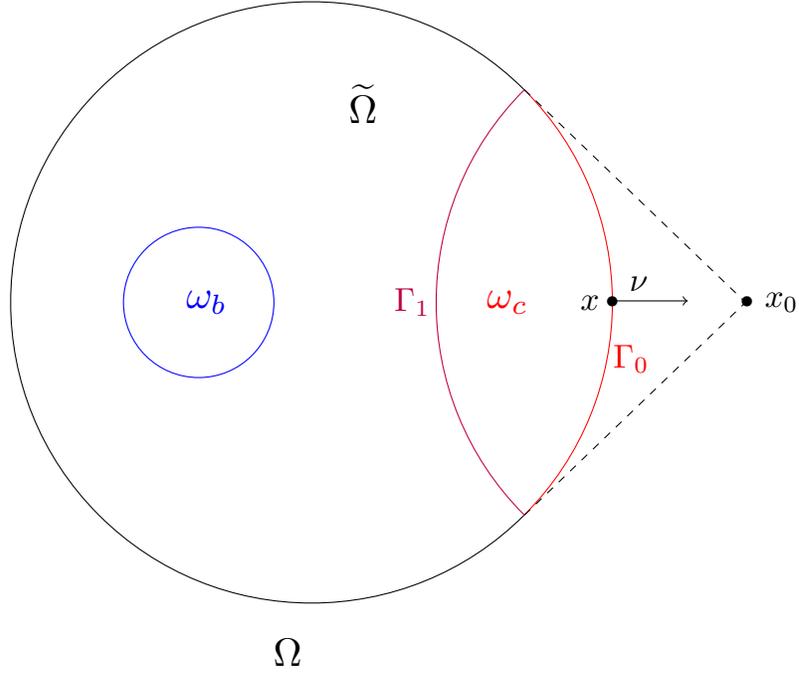
\noindent Let $x_0\in \mathbb{R}^{d}$ and $m(x)=x-x_0$ and suppose that (see Figure \ref{p7-Fig4})
\begin{equation}\tag{${\rm GC}$}\label{Geometric Condition}
m\cdot \nu\leq 0\quad \text{on}\quad \Gamma_0=\left(\partial\omega_c\right)-\Gamma_1.
\end{equation}
The main result of this section is the following theorem
\begin{theoreme}\label{Strong-Stability-ND}
Assume that \eqref{Geometric Condition} holds and 
\begin{equation}\label{GC-Condition}\tag{${\rm SSC}$}
\|c\|_{\infty}\leq \min\left\{\frac{1}{\|m\|_{\infty}+\frac{d-1}{2}},\frac{1}{\|m\|_{\infty}+\frac{(d-1)C_{p,\omega_c}}{2}}\right\},
\end{equation}
where $C_{p,\omega_c}$ is the Poincarr\'e constant on $\omega_c$. Then, the $C_0-$semigroup of contractions $\left(e^{t\mathcal{A}_d}\right)$ is strongly stable in $\mathcal{H}$; i.e. for all $U_0\in \mathcal{H}$, the solution of \eqref{ND-evolution} satisfies 
$$
\lim_{t\to +\infty}\|e^{t\mathcal{A}_d}U_0\|_{\mathcal{H}}=0.
$$
\end{theoreme}
\begin{proof}
First, let us prove that \begin{equation}\label{ker}\ker (i\la I-\AA_d)=\{0\},\ \forall \la \in \R.\end{equation} Since $0\in \rho(\mathcal{A}_d)$, then we still need to show the result for $\lambda\in \mathbb{R}^{\ast}$. Suppose that there exists a real number $\lambda\neq 0$ and $U=(u,v,y,z)^\top\in D(\mathcal{A}_d)$, such that 
$$
\mathcal{A}_dU=i\la U.
$$
Equivalently, we have 
\begin{eqnarray}
v&=&i\la u,\label{ND-ST1}\\
\divv(\nabla u+b\nabla v)-cz&=&i\la v,\label{ND-ST2}\\
z&=&i\la y, \label{ND-ST3}\\
\Delta y+cv&=&i\la z.\label{ND-ST4}
\end{eqnarray}
Next, a straightforward computation gives 
$$
0=\Re\left<i\la U,U\right>_{\mathcal{H}}=\Re\left<\mathcal{A}_dU,U\right>_{\mathcal{H}}=-\int_{\Omega}b\abs{\nabla v}^2dx,
$$
consequently, we deduce that 
\begin{equation}\label{ND-ST5}
b\nabla v=0\ \ \text{in}\ \ \Omega \quad \text{and}\quad \nabla v= \nabla u=0 \quad \text{in}\quad \omega_b.
\end{equation}
 Inserting \eqref{ND-ST1} in \eqref{ND-ST2}, then using the definition of $c$, we get 
\begin{equation}\label{ND-ST6}
\Delta u=-\la^2 u\quad \text{in}\quad \omega_b.
\end{equation}
From \eqref{ND-ST5} we get $\Delta u=0$ in $\omega_b$ and from  \eqref{ND-ST6} and the fact that $\la\neq 0$, we get 
\begin{equation}\label{ND-ST7}
u=0\quad \text{in}\quad \omega_b.
\end{equation}
Now, inserting \eqref{ND-ST1} in \eqref{ND-ST2},  then using \eqref{ND-ST5}, \eqref{ND-ST7} and the definition of $c$, we get 
\begin{equation}\label{ND-ST8}
\begin{array}{rll}
\la^2u+\Delta u&=&0\ \ \text{in}\ \ \widetilde{\Omega},\\
u&=&0\ \ \text{in}\ \ \omega_b\subset \widetilde{\Omega}.
\end{array}
\end{equation}
Using Holmgren uniqueness theorem, we get
\begin{equation}\label{ND-ST9}
u=0\quad \text{in}\quad \widetilde{\Omega}. 
\end{equation}
It follows that
\begin{equation}\label{ND-ST10}
u=\frac{\partial u}{\partial\nu}=0\quad \text{on}\quad \Gamma_1. 
\end{equation}
Now, our aim is to show that $u=y=0$ in $\omega_c$. For this aim, inserting \eqref{ND-ST1} and \eqref{ND-ST3} in \eqref{ND-ST2} and \eqref{ND-ST4}, then using \eqref{ND-ST5}, we get the following system 
\begin{eqnarray}
\la^2u+\Delta u-i\la cy&=&0\quad \text{in}\ \Omega,\label{ND-ST11}\\
\la^2y+\Delta y+i\la cu&=&0\quad \text{in}\ \Omega,\label{ND-ST12}\\
u&=&0\quad \text{on}\ \partial\omega_c,\label{ND-ST13}\\
y&=&0\quad \text{on}\ \Gamma_0,\label{ND-ST14}\\
\frac{\partial u}{\partial \nu}&=&0\quad \text{on}\ \Gamma_1.\label{ND-ST15}
\end{eqnarray}
Let us prove \eqref{ker} by the following three steps:\\\linebreak 
\textbf{Step 1.} The aim of this step is to show that 
\begin{equation}\label{ND-Step1-1}
\int_{\Omega}c\abs{u}^2dx=\int_{\Omega}c\abs{y}^2dx.
\end{equation}
For this aim, multiplying \eqref{ND-ST11} and \eqref{ND-ST12} by $\bar{y}$ and $\bar{u}$ respectively, integrating over $\Omega$ and using Green's formula, we get 
\begin{eqnarray}
\la^2\int_{\Omega}u\bar{y}dx-\int_{\Omega}\nabla u\cdot \nabla{\bar{y}}dx-i\la \int_{\Omega}c\abs{y}^2dx&=&0,\label{ND-Step1-2}\\
\la^2\int_{\Omega}y\bar{u}dx-\int_{\Omega}\nabla y\cdot \nabla{\bar{u}}dx+i\la \int_{\Omega}c\abs{u}^2dx&=&0.\label{ND-Step1-3}
\end{eqnarray}
Adding \eqref{ND-Step1-2} and \eqref{ND-Step1-3}, then taking the imaginary part, we get \eqref{ND-Step1-1}.\\
\noindent \textbf{Step 2.} The aim of this step is to prove the following identity
\begin{equation}\label{ND-Stpe2-1}
-d\int_{\omega_c}\abs{\la u}^2dx+(d-2)\int_{\omega_c}\abs{\nabla u}^2dx+\int_{\Gamma_0}(m\cdot \nu)\left|\frac{\partial u}{\partial\nu}\right|^2d\Gamma -2\Re\left(i\la \int_{\omega_c}cy\left(m\cdot \nabla{\bar{u}}\right)dx\right)=0.
\end{equation}
For this aim, multiplying \eqref{ND-ST11} by $2(m\cdot\nabla\bar{u})$, integrating over $\omega_c$ and taking the real part, we get 
\begin{equation}\label{ND-Stpe2-2}
2\Re\left(\la^2\int_{\omega_c}u(m\cdot \nabla\bar{u})dx\right)+2\Re\left(\int_{\omega_c}\Delta u(m\cdot \nabla\bar{u})dx\right)-2\Re\left(i\la\int_{\omega_c}cy(m\cdot\nabla\bar{u})dx\right)=0. 
\end{equation}
Now, using the fact that $u=0$ in $\partial\omega_c$, we get 
\begin{equation}\label{ND-Stpe2-3}
\Re\left(2\la^2\int_{\omega_c}u(m\cdot\nabla\bar{u})dx\right)=-d\int_{\omega_c}\abs{\la u}^2dx.  
\end{equation}
Using Green's formula, we obtain 
\begin{equation}\label{ND-Stpe2-4}
\begin{array}{ll}
\displaystyle
2\Re\left(\int_{\omega_c}\Delta u(m\cdot \nabla\bar{u})dx\right)=\displaystyle
-2\Re\left(\int_{\omega_c}\nabla u\cdot\nabla\left(m\cdot\nabla\bar{u}\right)dx\right)+2\Re\left(\int_{\Gamma_0}\frac{\partial u}{\partial\nu}\left(m\cdot\nabla\bar{u}\right)d\Gamma\right)\\[0.1in]
\hspace{3.85cm}=\displaystyle
(d-2)\int_{\omega_c}\abs{\nabla u}^2dx-\int_{\partial\omega_c}(m\cdot \nu)\abs{\nabla u}^2dx+2\Re\left(\int_{\Gamma_0}\frac{\partial u}{\partial\nu}\left(m\cdot\nabla\bar{u}\right)d\Gamma\right).
\end{array}
\end{equation}
Using \eqref{ND-ST13} and \eqref{ND-ST15}, we get 
\begin{equation}\label{ND-Stpe2-5}
\int_{\partial\omega_c}(m\cdot \nu)\abs{\nabla u}^2dx=\int_{\Gamma_0}(m\cdot\nu)\left|\frac{\partial u}{\partial\nu}\right|^2d\Gamma\ \ \text{and}\ \  \Re\left(\int_{\Gamma_0}\frac{\partial u}{\partial\nu}\left(m\cdot\nabla\bar{u}\right)d\Gamma\right)=\int_{\Gamma_0}(m\cdot\nu)\left|\frac{\partial u}{\partial\nu}\right|^2d\Gamma. 
\end{equation}
Inserting \eqref{ND-Stpe2-5} in \eqref{ND-Stpe2-4}, we get 
\begin{equation}\label{ND-Stpe2-6}
2\Re\left(\int_{\omega_c}\Delta u(m\cdot \nabla\bar{u})dx\right)=(d-2)\int_{\omega_c}\abs{\nabla u}^2dx+\int_{\Gamma_0}(m\cdot\nu)\left|\frac{\partial u}{\partial\nu}\right|^2d\Gamma. 
\end{equation}
Inserting \eqref{ND-Stpe2-3} and \eqref{ND-Stpe2-6} in \eqref{ND-Stpe2-2}, we get \eqref{ND-Stpe2-1}. \\\linebreak
\noindent \textbf{Step 3}. In this step, we prove \eqref{ker}.  Multiplying \eqref{ND-ST11} by $(d-1)\overline{u}$, integrating over $\omega_c$ and using \eqref{ND-ST13}, we get 
\begin{equation}\label{ND-Stpe2-7}
(d-1)\int_{\omega_c}\abs{\la u}^2dx+(1-d)\int_{\omega_c}\abs{\nabla u}^2dx-\Re\left(i\la (d-1)\int_{\omega_c}cy\bar{u}dx\right)=0.
\end{equation}
 Adding \eqref{ND-Stpe2-1} and \eqref{ND-Stpe2-7}, we get 
\begin{equation*}
\int_{\omega_c}\abs{\la u}^2dx+\int_{\omega_c}\abs{\nabla u}^2dx=\int_{\Gamma_0}(m\cdot \nu)\left|\frac{\partial u}{\partial\nu}\right|^2d\Gamma-2\Re\left(i\la \int_{\omega_c}cy\left(m\cdot \nabla{\bar{u}}\right)dx\right)-\Re\left(i\la (d-1)\int_{\omega_c}cy\bar{u}dx\right)=0.
\end{equation*}
Using \eqref{Geometric Condition}, we get 
\begin{equation}\label{ND-Stpe2-8}
\int_{\omega_c}\abs{\la u}^2dx+\int_{\omega_c}\abs{\nabla u}^2dx\leq 2\abs{\la}\int_{\omega_c}\abs{c}\abs{y}\abs{m\cdot \nabla u}dx+\abs{\la}(d-1)\int_{\omega_c}\abs{c}\abs{y}\abs{u}dx.
\end{equation}
Using Young's inequality and \eqref{ND-Step1-1}, we get 
\begin{equation}\label{ND-Stpe2-9}
2\abs{\la}\int_{\omega_c}\abs{c}\abs{y}\abs{m\cdot \nabla u}dx\leq \|m\|_{\infty}\|c\|_{\infty}\int_{\omega_c}\left(\abs{\la u}^2+\abs{\nabla u}^2\right)dx
\end{equation}
and 
\begin{equation}\label{ND-Stpe2-10}
\abs{\la}(d-1)\int_{\omega_c}\abs{c(x)}\abs{y}\abs{u}dx\leq \frac{(d-1)\|c\|_{\infty}}{2}\int_{\omega_c}\abs{\la u}^2dx+\frac{(d-1)\|c\|_{\infty}C_{p,\omega_c}}{2}\int_{\omega_c}\abs{\nabla u}^2dx.
\end{equation}
Inserting \eqref{ND-Stpe2-10} in \eqref{ND-Stpe2-8}, we get 
\begin{equation*}
\left(1-\|c\|_{\infty}\left(\|m\|_{\infty}+\frac{d-1}{2}\right)\right)\int_{\omega_c}\abs{\la u}^2dx+\left(1-\|c\|_{\infty}\left(\|m\|_{\infty}+\frac{(d-1)C_{p,\omega_c}}{2}\right)\right)\int_{\omega_c}\abs{\nabla u}^2dx\leq 0.
\end{equation*}
Using \eqref{GC-Condition} and \eqref{ND-Step1-1} in the above estimation, we get 
\begin{equation}\label{ND-Stpe2-11}
u=0\quad \text{and}\quad y=0\quad \text{in}\quad \omega_c.
\end{equation}
In order to complete this proof, we need to show that $y=0$ in $\widetilde{\Omega}$. For this aim, using the definition of the function $c$ in $\widetilde{\Omega}$ and using the fact that $y=0$ in $\omega_c$, we get 
\begin{equation}\label{ND-Stpe2-12}
\begin{array}{rll}
\displaystyle \la^2y+\Delta y&=&0\ \ \text{in}\ \ \widetilde{\Omega},\\[0.1in]
\displaystyle y&=&0 \ \ \text{on}\ \ \partial\widetilde{\Omega},\\[0.1in]
\displaystyle \frac{\partial y}{\partial \nu}&=&0\ \ \text{on}\ \ \Gamma_1.
\end{array}
\end{equation}
Now, using Holmgren uniqueness theorem, we obtain $y=0$ in $\widetilde{\Omega}$ and consequently \eqref{ker} holds true. Moreover, similar to Lemma 2.5 in \cite{akil2021ndimensional}, we can prove $R(i\la I-\AA_d)=\HH, \ \forall \la \in \R$. Finally, by using the closed graph theorem of Banach and Theorem \ref{App-Theorem-A.2}, we conclude the proof of this Theorem.
\end{proof}\\\linebreak

\noindent  Let us notice that, under the sole assumptions  \eqref{Geometric Condition}   
and \eqref{GC-Condition}, the polynomial stability   of system \eqref{ND-1}-\eqref{ND-5}
is an open problem.


\appendix
\section{Some notions and stability theorems}\label{p2-appendix}
\noindent In order to make this paper more self-contained, we recall in this short appendix some notions and stability results used in this work. 
\begin{definition}
\label{App-Definition-A.1}{\rm
Assume that $A$ is the generator of $C_0-$semigroup of contractions $\left(e^{tA}\right)_{t\geq0}$ on a Hilbert space $H$. The $C_0-$semigroup $\left(e^{tA}\right)_{t\geq0}$ is said to be 
		\begin{enumerate}
			\item[$(1)$] Strongly stable if 
			$$
			\lim_{t\to +\infty} \|e^{tA}x_0\|_H=0,\quad \forall\, x_0\in H.
			$$
			\item[$(2)$] Exponentially (or uniformly) stable if there exists two positive constants $M$ and $\varepsilon$ such that 
			$$
			\|e^{tA}x_0\|_{H}\leq Me^{-\varepsilon t}\|x_0\|_{H},\quad \forall\, t>0,\ \forall\, x_0\in H.
			$$
			\item[$(3)$] Polynomially stable if there exists two positive constants $C$ and $\alpha$ such that 
			$$
			\|e^{tA}x_0\|_{H}\leq Ct^{-\alpha}\|A x_0\|_{H},\quad \forall\, t>0,\ \forall\, x_0\in D(A).
			$$
			\xqed{$\square$}
	\end{enumerate}}
\end{definition}
\noindent To show  the strong stability of the $C_0$-semigroup $\left(e^{tA}\right)_{t\geq0}$ we rely on the following result due to Arendt-Batty \cite{Arendt01}. 
\begin{theoreme}\label{App-Theorem-A.2}{\rm
		{Assume that $A$ is the generator of a C$_0-$semigroup of contractions $\left(e^{tA}\right)_{t\geq0}$  on a Hilbert space $H$. If $A$ has no pure imaginary eigenvalues and  $\sigma\left(A\right)\cap i\mathbb{R}$ is countable,
			where $\sigma\left(A\right)$ denotes the spectrum of $A$, then the $C_0$-semigroup $\left(e^{tA}\right)_{t\geq0}$  is strongly stable.}\xqed{$\square$}}
\end{theoreme}
\noindent Concerning the characterization of polynomial stability stability of a $C_0-$semigroup of contraction $\left(e^{tA}\right)_{t\geq 0}$ we rely on the following result due to Borichev and Tomilov \cite{Borichev01} (see also \cite{Batty01} and \cite{RaoLiu01})
\begin{theoreme}\label{bt}
	{\rm
		Assume that $A$ is the generator of a strongly continuous semigroup of contractions $\left(e^{tA}\right)_{t\geq0}$  on $\mathcal{H}$.   If   $ i\mathbb{R}\subset \rho(\mathcal{A})$, then for a fixed $\ell>0$ the following conditions are equivalent
		\begin{equation}\label{h1}
		\limsup_{\la \in \R, \ |\la| \to \infty}\frac{1}{|\la|^\ell}\left\|(i\la I-\AA)^{-1}\right\|_{\mathcal{L}(\mathcal{H})}<\infty,
		\end{equation}
		\begin{equation}\label{h2}
		\|e^{t\mathcal{A}}U_{0}\|^2_{\HH} \leq \frac{C}{t^{\frac{2}{\ell}}}\|U_0\|^2_{D(\AA)},\hspace{0.1cm}\forall t>0,\hspace{0.1cm} U_0\in D(\AA),\hspace{0.1cm} \text{for some}\hspace{0.1cm} C>0.
		\end{equation}\xqed{$\square$}}
\end{theoreme}


\begin{thebibliography}{10}

\bibitem{Wehbe08}
F.~Abdallah, M.~Ghader, and A.~Wehbe.
\newblock Stability results of a distributed problem involving {B}resse system
  with history and/or {C}attaneo law under fully {D}irichlet or mixed boundary
  conditions.
\newblock {\em Math. Methods Appl. Sci.}, 41(5):1876--1907, 2018.

\bibitem{ABBresse}
M.~Akil and H.~Badawi.
\newblock The influence of the physical coefficients of a {B}resse system with
  one singular local viscous damping in the longitudinal displacement on its
  stabilization.
\newblock {\em Evolution Equations \& Control Theory}, 2022.

\bibitem{Akil2021}
M.~Akil, H.~Badawi, S.~Nicaise, and A.~Wehbe.
\newblock On the stability of {B}resse system with one discontinuous local
  internal {K}elvin--{V}oigt damping on the axial force.
\newblock {\em Zeitschrift f{\"u}r angewandte Mathematik und Physik},
  72(3):126, May 2021.

\bibitem{ABNWmemory}
M.~Akil, H.~Badawi, S.~Nicaise, and A.~Wehbe.
\newblock Stability results of coupled wave models with locally memory in a
  past history framework via nonsmooth coefficients on the interface.
\newblock {\em Mathematical Methods in the Applied Sciences}, 44(8):6950--6981,
  2021.

\bibitem{ABWdelay}
M.~Akil, H.~Badawi, and A.~Wehbe.
\newblock Stability results of a singular local interaction
  elastic/viscoelastic coupled wave equations with time delay.
\newblock {\em Communications on Pure \& Applied Analysis}, 20(9):2991--3028,
  2021.

\bibitem{Akil2020}
M.~Akil, Y.~Chitour, M.~Ghader, and A.~Wehbe.
\newblock Stability and exact controllability of a {T}imoshenko system with
  only one fractional damping on the boundary.
\newblock {\em Asymptotic Analysis}, 119:221--280, 2020.
\newblock 3-4.

\bibitem{akil2021ndimensional}
M.~Akil, I.~Issa, and A.~Wehbe.
\newblock A {N}-dimensional elastic/viscoelastic transmission problem with
  {K}elvin-{V}oigt damping and non smooth coefficient at the interface, 2021.

\bibitem{Alabau-Leautaud:13}
F.~Alabau-Boussouira and M.~L\'{e}autaud.
\newblock Indirect controllability of locally coupled wave-type systems and
  applications.
\newblock {\em J. Math. Pures Appl. (9)}, 99(5):544--576, 2013.

\bibitem{Arendt01}
W.~Arendt and C.~J.~K. Batty.
\newblock Tauberian theorems and stability of one-parameter semigroups.
\newblock {\em Trans. Amer. Math. Soc.}, 306(2):837--852, 1988.

\bibitem{BASSAM20151177}
M.~Bassam, D.~Mercier, S.~Nicaise, and A.~Wehbe.
\newblock Polynomial stability of the {T}imoshenko system by one boundary
  damping.
\newblock {\em Journal of Mathematical Analysis and Applications}, 425(2):1177
  -- 1203, 2015.

\bibitem{Batty01}
C.~J.~K. Batty and T.~Duyckaerts.
\newblock \href {http://dx.doi.org/10.1007/s00028-008-0424-1} {Non-uniform
  stability for bounded semi-groups on {B}anach spaces}.
\newblock {\em J. Evol. Equ.}, 8(4):765--780, 2008.

\bibitem{Borichev01}
A.~Borichev and Y.~Tomilov.
\newblock Optimal polynomial decay of functions and operator semigroups.
\newblock {\em Math. Ann.}, 347(2):455--478, 2010.

\bibitem{BurqSun:22}
N.~Burq and C.~Sun.
\newblock Decay rates for {K}elvin-{V}oigt damped wave equations {II}: {T}he
  geometric control condition.
\newblock {\em Proc. Amer. Math. Soc.}, 150(3):1021--1039, 2022.

\bibitem{FATORI2012600}
L.~H. Fatori and R.~N. Monteiro.
\newblock The optimal decay rate for a weak dissipative {B}resse system.
\newblock {\em Applied Mathematics Letters}, 25(3):600 -- 604, 2012.

\bibitem{Fatori01}
L.~H. Fatori, R.~N. Monteiro, and H.~D.~F. Sare.
\newblock \href {https://doi.org/10.1016/j.amc.2013.11.054} {The {T}imoshenko
  system with history and {C}attaneo law}.
\newblock {\em Applied Mathematics and Computation}, 228:128--140, Feb. 2014.

\bibitem{Gerbietal:21}
S.~Gerbi, C.~Kassem, A.~Mortada, and A.~Wehbe.
\newblock Exact controllability and stabilization of locally coupled wave
  equations: theoretical results.
\newblock {\em Z. Anal. Anwend.}, 40(1):67--96, 2021.

\bibitem{Hayek}
A.~Hayek, S.~Nicaise, Z.~Salloum, and A.~Wehbe.
\newblock A transmission problem of a system of weakly coupled wave equations
  with {K}elvin-{V}oigt dampings and non-smooth coefficient at the interface.
\newblock {\em SeMA J.}, 77(3):305--338, 2020.

\bibitem{Kassemetal:19}
C.~Kassem, A.~Mortada, L.~Toufayli, and A.~Wehbe.
\newblock Local indirect stabilization of {N}-d system of two coupled wave
  equations under geometric conditions.
\newblock {\em C. R. Math. Acad. Sci. Paris}, 357(6):494--512, 2019.

\bibitem{Kato01}
T.~Kato.
\newblock {\em Perturbation Theory for Linear Operators}.
\newblock Springer Berlin Heidelberg, 1995.

\bibitem{Liu-Rao:06}
K.~Liu and B.~Rao.
\newblock Exponential stability for the wave equations with local
  {K}elvin-{V}oigt damping.
\newblock {\em Z. Angew. Math. Phys.}, 57(3):419--432, 2006.

\bibitem{RaoLiu01}
Z.~Liu and B.~Rao.
\newblock Characterization of polynomial decay rate for the solution of linear
  evolution equation.
\newblock {\em Z. Angew. Math. Phys.}, 56(4):630--644, 2005.

\bibitem{Pazy01}
A.~Pazy.
\newblock {\em Semigroups of linear operators and applications to partial
  differential equations}, volume~44 of {\em Applied Mathematical Sciences}.
\newblock Springer-Verlag, New York, 1983.

\bibitem{Tebou:16}
L.~Tebou.
\newblock Stabilization of some elastodynamic systems with localized
  {K}elvin-{V}oigt damping.
\newblock {\em Discrete Contin. Dyn. Syst.}, 36(12):7117--7136, 2016.

\bibitem{Wehbe2021}
A.~Wehbe, I.~Issa, and M.~Akil.
\newblock Stability results of an elastic/viscoelastic transmission problem of
  locally coupled waves with non smooth coefficients.
\newblock {\em Acta Applicandae Mathematicae}, 171(1):23, Feb 2021.

\end{thebibliography}

\end{document}